\documentclass{article}
\usepackage[utf8]{inputenc}
\DeclareMathAlphabet{\mathpzc}{OT1}{pzc}{m}{it}

\usepackage{amsmath} 
\usepackage{amsfonts}
\usepackage{titlesec}
\usepackage{dsfont}
\usepackage[utf8]{inputenc}
\usepackage[english]{babel}
\usepackage[T1]{fontenc}
\usepackage{physics}
\usepackage{authblk}
\usepackage{tikz}
\usepackage{subcaption}
\usepackage{graphicx}
\usepackage{multicol}
\usepackage{amsthm}
\newtheorem{theorem}{Theorem}[section]
\newtheorem{proposition}{Proposition}[section]
\newtheorem{corollary}{Corollary}[theorem]
\newtheorem{lemma}[theorem]{Lemma}
\newtheorem*{remark}{Remark}
\usepackage{anyfontsize}
\newmuskip\pFqmuskip
\newcommand*\pFq[6][8]{%
  \begingroup 
  \pFqmuskip=#1mu\relax
  \mathcode`\,=\string"8000
  \begingroup\lccode`\~=`\,
  \lowercase{\endgroup\let~}\pFqcomma
  {}_{#2}F_{#3}{\left[\genfrac..{0pt}{}{#4}{#5};#6\right]}%
  \endgroup
}
\newcommand{\pFqcomma}{\mskip\pFqmuskip}
\makeatletter
\newsavebox{\@brx}
\newcommand{\llangle}[1][]{\savebox{\@brx}{\(\m@th{#1\langle}\)}%
  \mathopen{\copy\@brx\kern-0.5\wd\@brx\usebox{\@brx}}}
\newcommand{\rrangle}[1][]{\savebox{\@brx}{\(\m@th{#1\rangle}\)}%
  \mathclose{\copy\@brx\kern-0.5\wd\@brx\usebox{\@brx}}}
\makeatother

\begin{document}

\title{The Terwilliger algebra of symplectic dual polar graphs, the subspace lattices and $U_q(sl_2)$}
\author[1]{Pierre-Antoine Bernard}
\author[2]{Nicolas Crampé}
\author[1,3]{Luc Vinet}
\affil[1]{Centre de recherches mathématiques, Université de Montréal, P.O. Box 6128, Centre-ville
Station, Montréal (Québec), H3C 3J7, Canada,}
\affil[2]{Institut Denis-Poisson CNRS/UMR 7013 - Université de Tours - Université d’Orléans, Parc de
Grandmont, 37200 Tours, France.}
\affil[3]{IVADO, 6666 Rue Saint-Urbain, Montréal (Québec), H2S 3H1, Canada}

\maketitle
\begin{abstract}
The adjacency matrix of a symplectic dual polar graph restricted to the eigenspaces of an abelian automorphism subgroup is shown to act as the adjacency matrix of a weighted subspace lattice. The connection between the latter and $U_q(sl_2)$ is used to find the irreducible components of the standard module of the Terwilliger algebra of symplectic dual polar graphs. The multiplicities of the isomorphic submodules are given. 

\bigskip


\end{abstract}
\section{Introduction}
$P$- polynomial association schemes arise in the description of the neighborhoods of vertices in distance-regular graphs. They play an important role in combinatorics, coding theory \cite{Bannai1984AlgebraicCI, brouwer2012distance} and have found applications in the study of quantum systems \cite{Hamming, bernard2021entanglement, chan2019fractional,crampHad}. In the case of schemes that are also $Q$- polynomial, Leonard's theorem \cite{leonard1982orthogonal} implies that they are related to the hypergeometric orthogonal polynomials of the Askey-scheme \cite{koekoek1996askey}. It is well known, for instance, that the Hamming and Johnson graphs are related to the Krawtchouk and dual Hahn polynomials \cite{Bannai1984AlgebraicCI} respectively. 

The Terwilliger algebra $\mathcal{T}$ was introduced in \cite{thini, thinii, thiniii} to study $P$- and $Q$-polynomial association schemes. The thin irreducible modules of this semi-simple algebra form Leonard systems \cite{TERWILLIGER2003463}, which give a convenient way of characterizing $\mathcal{T}$. In recent years, much efforts have been dedicated to the decomposition of the standard module of distance-regular graphs in irreducible submodules of their Terwilliger algebra. The Hamming \cite{Hamming,GO2002399,huang2021clebschgordan,LEVSTEIN20061} and Johnson \cite{bernard2021entanglement,GAO2014164,LEVSTEIN20071621,Obse,TAN2019157} cases have been worked out in great details. Distance-regular graphs associated to certain $q$-polynomials of the Askey-scheme have also received some attention. The Terwilliger algebra of the Grassmann graphs, which are related to dual $q$-Hahn polynomials, has been investigated in \cite{GAO2015427,LIANG2020117}. In this paper, we pursue these questions and consider the Terwilliger algebra of distance-regular graphs associated to dual $q$-Krawtchouk polynomials. In particular, we focus on the case of symplectic dual polar graphs \cite{brouwer2012distance}.

Dual polar graphs were introduced by Stanton in \cite{stanton1980some} as $q$-deformations of hypercubes. It was shown in \cite{WORAWANNOTAI2013443} that the Terwilliger algebra $\mathcal{T}$ of these graphs corresponds to some central extension of $U_q(sl_2)$. Furthermore, the irreducible $\mathcal{T}$-submodules were identified as Leonard systems of dual $q$-Krawtchouk type. 

While these results give a good description of $\mathcal{T}$, the characterization remains incomplete. Some parameters $(r,t,d)$ on which depends the exact nature of the dual $q$-Krawtchouk Leonard systems have not been determined in \cite{WORAWANNOTAI2013443}. Identifying these parameters for each irreducible submodule of the dual polar graphs of type $[C_D(q)]$ with $q$ prime constitutes our central result. We also obtain the multiplicity of the submodules for each set of parameters. 

Some central operators were introduced in \cite{WORAWANNOTAI2013443} to relate $\mathcal{T}$ and the quantum group $U_q(sl_2)$. Since their explicit construction depends on the identification of $(r,t,d)$ for each submodule, our results also shed light on the relation between these two algebraic structures. 

The paper is organized as follows. In section \ref{s2}, we recall the definitions of $P$- and $Q$-polynomial association schemes and of their Terwilliger algebra. In section \ref{s3}, we present the construction of the symplectic dual polar graphs and describe some of their properties. We look at the automorphism group of these graphs and its action on the standard module. We also highlight the existence of a special abelian subgroup $H$. In section \ref{s5}, we construct a basis of the standard module which diagonalizes the action of the group elements in $H$. In section \ref{sbonus}, we show that there is a one-to-one correspondence between the vectors of this basis and the characteristic vectors of a subspace lattice $L_N(q)$. In section \ref{s6}, we present the subspace lattice graph $L_N(q)$ and its relation with $U_{\sqrt{q}}(sl_2)$. In section \ref{s7}, we show that the restriction of the adjacency matrix of a dual polar graph of type $[C_D(q)]$ to an eigenspace of $H$ corresponds to the adjacency matrix of a weighted subspace lattice. We use this connection and the relation between $L_N(q)$ and $U_{\sqrt{q}}(sl_2)$ to obtain the irreducible $\mathcal{T}-$submodules. We give the parameters $(r,t,d)$ and the multiplicity for the isomorphic submodules.

\titleformat*{\section}{\fontsize{14}{14}\bfseries}
\section{The Terwilliger algebra of $P$- and $Q$- polynomial association schemes}

\titleformat*{\section}{\LARGE\bfseries}
\label{s2}

Let $X$ be a set and $\{R_i\}_{0 \leq i \leq D}$ be a partition of $X \times X$, the set of all possible edges connecting two elements in $X$. It is said that $(X,\{R_i\}_{0 \leq i \leq D} )$ forms a \textit{symmetric association scheme} whenever
\begin{itemize}
    \item $R_0 = \{(x,x) : x \in X\}$;
    \item $(x,y) \in R_i \iff (y,x) \in R_i$;
    \item for any $x,y$ such that $(x,y) \in R_k$, the number $p_{ij}^k$ of $z \in X$ such that $(x,z) \in R_i$ and $(y,z) \in R_j$ depends only on $i,j$ and $k$.
\end{itemize}
The standard module $V$ of a symmetric association scheme is the vector space over $\mathbb{C}$ spanned by the characteristic vectors labeled by the elements of $X$, i.e.
\begin{align}
    V = \text{span}_{\mathbb{C}}\{\ket{x} : x \in X\},
\end{align}
where $\text{span}_K$ refers to the span over the field $K$. The $i^{\text{th}}$ adjacency matrix $A_i$ of a scheme is the matrix acting on $V$ with entries 
\begin{align}
    (A_i)_{xy} =  
\left\{
    \begin{array}{ll}
    	1  & \mbox{if } (x,y) \in R_i, \\
    	0 & \mbox{otherwise, }
    \end{array}
\right.
\label{defai}
\end{align}
where $(A_i)_{xy} = \bra{x} A_i \ket{y}$. These matrices verify the following relations: 
\begin{multicols}{2}
\begin{itemize}
        \item $A_0 = \mathds{1}$;
        \item $\displaystyle \sum_{i =0}^D A_i  = \mathbb{J}$;
        \item $A_i \circ A_j = \delta_{ij} A_i$; 
        \item $A_i A_j =  \displaystyle \sum_{k=0}^D p_{ij}^k A_k$,
\end{itemize}
\end{multicols}
\noindent where $\circ$ is the entry-wise product, $\mathbb{J}$ is the $|X| \times |X|$ matrix of ones and the intersection numbers $p_{ij}^k$ are positive integers. The matrices $A_i$ commute and thus share a common set of eigenspaces. The number of different eigenspaces is $D+1$, which is also the number of adjacency matrices \cite{Bannai1984AlgebraicCI}. The projectors onto these spaces are denoted $E_i$, $0\leq i \leq D$, and verify:
\begin{multicols}{2}
 \begin{itemize}
        \item$E_0 = \displaystyle\frac{1}{|X|} \mathbb{J}$;
        \item $\displaystyle\sum_{i =0}^D E_i  = \mathds{1}$;
        \item  $E_i E_j = \delta_{ij} E_i$ ; 
        \item  $E_i \circ E_j = \displaystyle\frac{1}{|X|}\displaystyle\sum_{k=0}^D q_{ij}^k E_k$,
    \end{itemize}
\end{multicols}
\noindent where the Krein parameters $q_{ij}^k$ are non-negative real numbers. The commutative algebra generated by the set of adjacency matrices $\{A_i\}_{0\leq i \leq D}$ is referred to as the \textit{Bose-Mesner algebra} of the scheme. 

Given an arbitrary choice of a reference vertex $x_0 \in X$, one can also construct dual adjacency matrices $A_i^*$ and dual projectors $E_i^*$. They are diagonal matrices acting on $V$ with entries
\begin{align}
    (A_i^*)_{xx} = |X| (E_{i})_{x_0 x},
    \label{defdualA}
\end{align}
and
\begin{align}
    (E_i^*)_{xx} = (A_{i})_{x_0 x} = 
    \left\{
    \begin{array}{ll}
    	1  & \mbox{if } (x_0,x) \in R_i, \\
    	0 & \mbox{otherwise. }
    \end{array}
\right.,
    \label{defdualP}
\end{align}
One can check that the dual adjacency matrices verify
\begin{align}
    A^*_{i} A^*_{j} = \sum_{k=0}^D q_{i j}^{k} A^*_{k}
\end{align}
and that the $E_i^*$ really act as projectors, i.e.
\begin{align}
    \sum_{i = 0}^{D} E_i^*  = \mathds{1}, \quad \quad E_{i}^* E_{j}^* = \delta_{i j}E_{i}^*.
\end{align}

A symmetric association scheme is said to be $P$-\textit{polynomial} when for each $i \in \{0,1, \dots, D\}$ there exists a polynomial $\gamma_i$ of degree $i$ such that 
\begin{align}
    A_i = \gamma_i(A_1).
    \label{ppoly}
\end{align}
Similarly, it is said to be $Q$-\textit{polynomial} when for each $i \in \{0,1, \dots, D\}$ there exists a polynomial $\gamma^*_i$ of degree $i$ such that 
\begin{align}
    A_i^* = \gamma^*_i(A_1^*).
    \label{qpoly}
\end{align}
Given a $P$- and $Q$- polynomial association scheme, one defines its \textit{Terwilliger algebra} $\mathcal{T}$ as the algebra spanned by its adjacency and dual adjacency matrices. Using \eqref{ppoly} and \eqref{qpoly}, one sees that it is generated by two elements, i.e.
\begin{align}
    \mathcal{T} = \langle A, A^* \rangle,
    \label{defter}
\end{align}
where we used the simplified notation $A = A_1$ and $A^* = A_1^*$.

\subsection{Automorphism group of an association scheme }

The automorphism group $G$ of an association scheme $(X, \{R_i\}_{0\leq i \leq D})$ has for elements the maps ${g} : X \rightarrow X$ verifying
\begin{align}
    (x,y) \in R_i \iff ({g}x,g y) \in R_i, \quad \forall x,y \in X
    \label{defauto}
\end{align}
and for group product the composition of maps. The representation $\rho$ of $G$ on the standard module is given by
\begin{align}
    \rho({g}) \ket{{x}} = \ket{gx},
\end{align}
such that $V$ is both a $\mathcal{T}$- and a $G$-module. Using \eqref{defai} and \eqref{defauto}, one finds that $\rho(g)$ commutes with the matrices in the Bose-Mesner algebra, i.e.
\begin{align}
    [A_i,\rho(g)] = 0, \quad \forall g \in G,\  \forall i \in \{0,1,\dots, D\}.
\end{align}
Restricting to a subgroup $P \subset G$ stabilizing $x_0$, one also gets that
\begin{align}
    [A_i^*,\rho(g)] = 0, \quad \forall g \in P,\  \forall i \in \{0,1,\dots, D\}.
\end{align}
It follows that any eigenspace W of $\rho(g)$ is stabilized by the Terwilliger algebra $\mathcal{T}$.

\section{Dual polar graphs of type $[C_D(q)]$}
\label{s3}
We recall here the definition of the symplectic dual polar graphs and their connection with association schemes and the group of Lie type $Sp(2D,q)$. While these graphs are defined for $q$ being any power of a prime number, we shall restrict ourselves to $q$ prime. 

\subsection{The definition of the graphs}
Dual polar graphs of type $[C_D(q)]$ (or symplectic) are constructed in the following way \cite{brouwer2012distance}. Let $\mathbb{F}_q^{2D}$ be a vector space of dimension $2D$ over the finite field $\mathbb{F}_q$ and equipped with a non-degenerate symplectic form $\mathfrak{B}$. There exists a symplectic basis $\{e_1, \dots, e_D\}\cup\{f_1, \dots, f_D\}$ such that
\begin{align}
    \mathfrak{B}(e_i, e_j) = \mathfrak{B}(f_i, f_j) = 0 \quad \text{and} \quad  \mathfrak{B}(e_i, f_j) =-\mathfrak{B}(f_j, e_i) = \delta_{ij}.
\end{align}
\noindent The set of vertices $X$ of a symplectic dual polar graph is the set of all the \textit{maximal isotropic subspaces} of $\mathbb{F}_q^{2D}$. In other words, the subspace $x \subset \mathbb{F}_q^{2D}$ is a vertex if
\begin{align}
    \mathfrak{B}(v_1,v_2) = 0, \quad \quad \forall v_1, v_2 \in x
\end{align}
and if $x$ is not contained in a larger isotropic subspace. By Witt's theorem, the dimension of a maximal isotropic subspace of $\mathbb{F}_q^{2D}$ is always equal to $D$ \cite{stanton1980some} and, one gets, for any $x\in X$, that
\begin{align}
    \text{dim}(x) = D.
\end{align}
For instance, the subspace of dimension $D$
\begin{align}
    x_i = \text{span}_{\mathbb{F}_q}\{e_1, \dots, e_{D-i}\}\  \oplus\  \text{span}_{\mathbb{F}_q}\{f_{D-i+1}, \dots, f_{D}\}
\end{align}
is a vertex for all $i \in \{0, 1, \dots, D\}$. In the graph, two maximal isotropic subspaces $x$ and $y$ are connected by an edge when 
\begin{align}
    \text{dim}(x \cap y) = D-1.
\end{align}
The distance between two vertices is therefore
\begin{align}
    \text{dist}(x,y) = D- \text{dim}(x \cap y).
    \label{defdista}
\end{align}
For example, 
\begin{align}
    \text{dist}(x_i,x_j) = |i-j|.
\end{align}
An important property of these graphs is that they are \textit{distance-regular}, i.e. the numbers
\begin{align}
    p_{ij}^k = |\{ z \in X : \text{dist}(x,z) = i, \text{dist}(y,z) = j  \}|
    \label{didi}
\end{align}
are the same for all $x, y \in X$ such that $\text{dist}(x,y) = k$. This implies that $(X, \{R_i\}_{0 \leq i \leq D})$ with
\begin{align}
    (x,y) \in R_i \iff \text{dist}(x,y) = i
\end{align}
forms a $P$-polynomial association scheme (chapter III in \cite{Bannai1984AlgebraicCI},  Proposition $1.1$) with intersection parameters $p_{ij}^k$ given by \eqref{didi}. In particular, its parameters $a_i \equiv p_{1i}^i$, $b_i \equiv p_{1i+1}^i$ and $c_i \equiv p_{1i-1}^i$ are given by \cite{brouwer2012distance}
\begin{align}
a_i = q^i - 1, \quad b_i = \frac{q^{i+1}(q^{D-i} - 1)}{q-1}, \quad c_i = \frac{q^{i} - 1}{q-1}.
\end{align}
This scheme is also known to be $Q$-polynomial \cite{Bannai1984AlgebraicCI} and thus equipped with a Terwilliger algebra $\mathcal{T}$, as defined in \eqref{defter}.
\subsection{Automorphism subgroups}
The full automorphism group of a dual polar graph of type $[C_D(q)]$ is the projective semilinear symplectic group (\cite{brouwer2012distance}, Theorem 9.4.3). A subgroup is obtained by restricting to maps associated to matrices in $Sp(2D,q)$, the group of Lie type composed of the non-singular $2D \times 2D$ matrices with entries in $\mathbb{F}_q$ that preserve $\mathfrak{B}$. Let $g \in Sp(2D,q)$ and $x=\text{span}_{\mathbb{F}_q}\{v_1, \dots, v_D\}\in X$. The action of $g$ on $x$ is defined by \begin{align}
    {g} \ \text{span}_{\mathbb{F}_q}\{v_1, \dots v_D\}=  \text{span}_{\mathbb{F}_q}\{g v_1, \dots g v_D\}
    \label{sp}
\end{align}
and provides an automorphism of the scheme. In most cases, the transformations in this subgroup are sufficient. For instance, they are enough to show that symplectic dual polar graphs are two-point homogeneous spaces and thus distance-regular:
\begin{theorem}
{\normalfont(\cite{stanton1980some}, Theorem 5.1)} The function $\text{dist}(y,z)$ is a metric on $X$ and if $\text{dist}(y_1,z_1) = \text{dist}(y_2, z_2)$, there exists $g \in Sp(2D,q)$ such that $({g}y_1, g z_1) = (y_2,z_2)$.
\label{th1}
\end{theorem}

For any matrix $g \in Sp(2D, q)$, one can define the $D \times D$ matrices $\mathcal{C}$, $\mathcal{G}$, $\mathcal{D}$ and $\mathcal{E}$ as the following blocks in $g$:
\begin{align}
g = 
    \begin{pmatrix}
\mathcal{C} & \mathcal{G}\\
\mathcal{D} & \mathcal{E}
\end{pmatrix},
\end{align}
\noindent where $\mathcal{C}$ and $\mathcal{E}$ act on the subspaces spanned by $\{e_1, e_2, \dots, e_{D-1}, e_D\}$  and $\{f_1, f_2 \dots, f_{D-1},f_D\}$ respectively. 
To preserve the form $\mathfrak{B}$ implies
\begin{align}
g^t
\begin{pmatrix}
0 & \mathds{1} \\
-\mathds{1} & 0
\end{pmatrix}
g = 
\begin{pmatrix}
0 & \mathds{1} \\
-\mathds{1} & 0
\end{pmatrix}.
\end{align}
\noindent In terms of $\mathcal{C}$, $\mathcal{G}$, $\mathcal{D}$ and $\mathcal{E}$, this amounts to
\begin{align}
    -\mathcal{C}^t \mathcal{D} + \mathcal{D}^t \mathcal{C} = 0, \quad -\mathcal{G}^t \mathcal{E} + \mathcal{E}^t \mathcal{G} = 0 \quad \text{and} \quad -\mathcal{D}^t\mathcal{G} + \mathcal{C}^t \mathcal{E} = \mathds{1}.
    \label{spcondi}
\end{align}
 Let us now consider the parabolic subgroup $P \subset Sp(2D,q)$ which corresponds to the stabilizer of $x_0 = \text{span}_{\mathbb{F}_q}\{e_1, \dots, e_D\}$. For $g \in P$, one finds that the block $\mathcal{D}$ must be zero. Then, \eqref{spcondi} leads to the following decomposition
\begin{align}
g = 
    \begin{pmatrix}
\mathcal{C} & 0\\
0 & (\mathcal{C}^t)^{-1}
\end{pmatrix}
    \begin{pmatrix}
\mathds{1}& \mathcal{F}\\
0 & \mathds{1}
\end{pmatrix},
\label{decomp}
\end{align}
\noindent where $\mathcal{C} \in GL(D, q)$ is an invertible matrix and $\mathcal{F} = \mathcal{C}^{-1}\mathcal{G} \in \text{Sym}_D $ is a symmetric matrix. In the following, we shall refer to an element in $P$ by the pair of matrices $(\mathcal{C},\mathcal{F})$ appearing in its decomposition \eqref{decomp}, i.e.
\begin{align}
    P = \left(\{(\mathcal{C},\mathcal{F}) : \mathcal{C} \in GL(D, q), \mathcal{F} \in \text{Sym}_D \}, \ \cdot \ \right),
\end{align}
where the group product $\cdot$ is given by the semi-direct product of $GL(D,q)$ with $\text{Sym}_D$:
\begin{align}
    (\mathcal{C}_1,\mathcal{F}_1) \cdot (\mathcal{C}_2,\mathcal{F}_2) = (\mathcal{C}_1\mathcal{C}_2,\  \mathcal{C}_2^{-1}\mathcal{F}_1(\mathcal{C}_2^t)^{-1} + \mathcal{F}_2).
\end{align}
Now, consider the following subgroup of $P$:
\begin{align}
    H = \left(\{(\mathds{1}, \mathcal{F}) : \mathcal{F} \in \text{Sym}_D\}, \ \cdot \ \right).
    \label{definitionH}
\end{align}
It is isomorphic to $(\mathbb{F}_q, + )^{ \frac{D(D+1)}{2}}$ and thus abelian. Since we restrict $q$ to being a prime, the subgroup $H$ therefore corresponds to $D(D+1)/2$ copies of the cyclic group of dimension $q$. 
\section{Eigenspaces of $H$}
\label{s5}
Since the subgroup $H$ is abelian, there exists a basis of the standard module $V$ which diagonalizes simultaneously the action of all the group elements $(\mathds{1}, \mathcal{F})$. The aim of this section is to construct this basis. Consider the following decomposition of $V$ in terms of neighborhoods of $x_0$:
\begin{align}
    V = \bigoplus_{i = 0}^D V_i,
    \label{decrom}
\end{align}
where 
\begin{align}
    V_i = \text{span}_{\mathbb{C}}\{ \ket{x} : x \in X \ \text{s.t.} \  \text{dist}(x,x_0) = i\}.
    \label{vi}
\end{align}
Taking $y_1 = y_2 = x_0$ in theorem \ref{th1}, one finds that the $i^{\text{th}}$ neighborhood of $x_0$ corresponds to the orbit of $P$ acting on any reference vertex $ref_i \in X$, $\text{dist}(x_0,ref_i) =i$. Without loss of generality, let us pick 
\begin{align}
    ref_i = x_i = \text{span}_{\mathbb{F}_q}\{e_1, \dots, e_{D-i}\}\  \oplus\  \text{span}_{\mathbb{F}_q}\{f_{D-i+1} \dots, f_{D}\},
\end{align}
and rewrite \eqref{vi} as
\begin{align}
    V_i = \rho(P) \ket{x_i} \equiv \text{span}_{\mathbb{C}}\{ \rho((\mathcal{C},\mathcal{F})) \ket{x_i} : (\mathcal{C},\mathcal{F}) \in P \}.
    \label{Vorb}
\end{align}
Since we have that
\begin{align}
    \rho((\mathds{1},\mathcal{F})) V_i \subseteq V_i, \quad \forall\ \mathcal{F} \in \text{Sym}_{D},
\end{align}
we can look for the eigenvectors of the matrices $\rho((\mathds{1},\mathcal{F}))$ in each of the submodules $V_i = \rho(P)\ket{x_i}$ separately. First, we consider those in the orbit $\rho(H) \ket{x_i}$ of $\ket{x_i}$ under $H$. We note that:

\begin{lemma}
Two group elements $(\mathds{1},\mathcal{F})$ and $(\mathds{1},\mathcal{F}')$ of $H$ have the same action on $\ket{x_i}$ if and only if $ \mathcal{F}_{mn} = \mathcal{F}_{mn}'$ $\forall \  m,n > D-i$.
\label{lementries}
\end{lemma}
\begin{proof}
 By definition, we have
 \begin{align}
     \rho((\mathds{1},\mathcal{F})) \ket{x_i} = \ket{(\mathds{1},\mathcal{F}) x_i}.
 \end{align}
 From the action of $H$ on the $\mathbb{F}_q^{2D}$, we get that
 \begin{equation}
     \begin{split}
          (\mathds{1}, \mathcal{F}) x_i  &= \text{span}_{\mathbb{F}_q}\{e_1, \dots, e_{D-i}\}\  \oplus\  \text{span}_{\mathbb{F}_q}\{(\mathds{1}, \mathcal{F}) f_{D-i+1}, \dots, (\mathds{1}, \mathcal{F}) f_{D}\}, \\
          & = x_i \cap x_0  \oplus\  \text{span}_{\mathbb{F}_q}\{ f_{D-i+1} + \sum_{j = 1}^D \mathcal{F}_{j, D-i+1} e_j,  \dots,  f_{D} + \sum_{j = 1}^D \mathcal{F}_{j, D} e_j\}.
          \label{masterac}
     \end{split}
 \end{equation}
 Changing the entries $\mathcal{F}_{mn}$ with $m \text{ or } n \leq D-i$ amounts to translate the basis vectors $f_{k} + \sum_{j = 1}^D \mathcal{F}_{j, k} e_j$ by a linear combination of vectors in $x_i \cap x_0$. Since the latter is always in the subspace $(\mathds{1},\mathcal{F})x_i$, this modification does not affect the resulting subspace.
 
 If we have $\mathcal{F}_{mn} \neq \mathcal{F}'_{mn}$ with $m,n > D -i$, we  find from equation \eqref{masterac} that $(\mathds{1}, \mathcal{F})x_i \neq (\mathds{1}, \mathcal{F}')x_i$. The lemma follows. 
\end{proof}

\begin{remark}
 Equation \eqref{masterac} also shows that every subspace in the orbit of $x_i$ under $H$ shares the same intersection with $x_0$.
\end{remark}

\noindent From lemma \ref{lementries}, we see that the basis vectors of $\rho(H)\ket{x_i}$ are all obtained only once by acting on $\ket{x_i}$ with the group elements $(\mathds{1}, \mathcal{F})$, with $\mathcal{F}$ having only non-zero entries in its bottom right $i\times i$ block, i.e
\begin{align}
    \mathcal{F}_{mn} = 0 \quad \text{  for  }\quad m \text{ or } n \leq D-i.
    \label{condi49}
\end{align}
We refer to the set of $D \times D$ symmetric matrices verifying \eqref{condi49} as $\text{Sym}_{D,i}$. Now, for a matrix $S \in \text{Sym}_{D,i}$, consider the following vector:
\begin{align}
    \ket{[\mathds{1}], {S}}_i = \sum_{\mathcal{F} \in \text{Sym}_{D,i}} \chi_S(\mathcal{F}) \ket{(\mathds{1},\mathcal{F}) x_i},
    \label{eigenH}
\end{align}
\noindent with 
\begin{equation}
    \begin{split}
        \chi_S(\mathcal{F}) &= q^{-\frac{i(i+1)}{4}}\prod_{j,k = 0}^D e^{\frac{ 2\pi i}{q} S_{jk}\mathcal{F}_{jk}} = q^{-\frac{i(i+1)}{4}} e^{\frac{ 2\pi i}{q} \text{tr}(S \mathcal{F})},
        \label{eigenH2}
    \end{split}
\end{equation}
where $\text{tr}(S \mathcal{F})$ refers to the trace of $S \mathcal{F}$. Given two vectors $\ket{[\mathds{1}], {S}}_i$ and $\ket{[\mathds{1}], {S'}}_i$, the condition $S,S' \in \text{Sym}_{D,i}$ implies that
\begin{align}
    \text{tr}((S' - S)\mathcal{F}) = 0, \quad \forall \mathcal{F} \in \text{Sym}_{D,i} \quad \iff \quad S = S'.
\end{align}
Thus, a direct computation shows that
\begin{align}
    \mathstrut_i\bra{[\mathds{1}],{S}}\ket{[\mathds{1}], {S}'}_i = \delta_{S,S'}.
\end{align}
\noindent One also finds
\begin{align}
    \rho((\mathds{1}, \mathcal{F})) \ket{[\mathds{1}],{S}}_i = \lambda_S(\mathcal{F}) \ket{[\mathds{1}],{S}}_i,
\end{align}
\noindent with
\begin{align}
    \lambda_{{S}}(\mathcal{F}) = e^{\frac{-2 \pi i}{q} \text{tr}(S \mathcal{F})}.
\end{align}
Therefore, $\{\ket{[\mathds{1}],S}_i : S \in \text{Sym}_{D,i} \}$ forms an orthonormal basis for the orbit  $\rho(H)\ket{x_i}$ which diagonalizes the action of the group elements of $H \cong \{(\mathds{1}, \mathcal{F})\}$.  It is worth noting that \eqref{eigenH} and \eqref{eigenH2} are motivated by the representation theory of cyclic groups and necessitate the restriction of $q$ to being a prime. When $q$ is a prime power, $H$ is no longer multiple copies of a cyclic group and other expressions are required.

To obtain a basis for $V_i$ defined by eigenvectors of the matrices $\rho((\mathds{1},\mathcal{F}))$, we also need to consider the action of the group elements $(\mathcal{C}, 0)$ on $\rho(H)\ket{x_i}$. Taking $S \in \text{Sym}_{D}$ such that $\mathcal{C}^t S \mathcal{C} \in \text{Sym}_{D,i}$ and using
\begin{align}
    (\mathds{1}, \mathcal{F})\cdot (\mathcal{C},0) =  (\mathcal{C},0)\cdot(\mathds{1}, \mathcal{C}^{-1}\mathcal{F}(\mathcal{C}^{-1})^t),
    \label{switch}
\end{align}
one finds that
\begin{align}
    \ket{[\mathcal{C}],S}_i \equiv \rho((\mathcal{C},0)) \ket{[\mathds{1}],\mathcal{C}^t S \mathcal{C}}_i
\end{align} is also an eigenvector of the matrices in $H$. Indeed, we have
\begin{equation}
    \begin{split}
         \rho((\mathds{1}, \mathcal{F})) \ket{[\mathcal{C}], S}_i &=  \lambda_{\mathcal{C}^tS\mathcal{C}}(\mathcal{C}^{-1}\mathcal{F}(\mathcal{C}^{-1})^t) \ket{[\mathcal{C}], S}_i \\ & =  \lambda_{S}(\mathcal{F}) \ket{\mathcal{C}, S}_i.
    \end{split}
\end{equation}
Here, $[\mathcal{C}]$ is the equivalence class of $\mathcal{C}$ with respect to the relation $\sim_{S,i}$ defined as
\begin{align}
    \mathcal{C} \sim_{S,i} \mathcal{C}' \iff \rho((\mathcal{C},0)) \ket{[\mathds{1}],\mathcal{C}^t S \mathcal{C}}_i = \rho((\mathcal{C}',0)) \ket{[\mathds{1}],\mathcal{C}'^t S \mathcal{C}'}_i.
\end{align}
 This takes into account that different matrices $\mathcal{C}$ can lead to the same vector. The following lemma gives the necessary and sufficient condition for two matrices to be equivalent in this respect.
\begin{lemma}
Two matrices $\mathcal{C}$ and $\mathcal{C}' \in GL(D,q)$ verify $\mathcal{C} \sim_{S,i} \mathcal{C}'$ if and only if
\begin{align}
   \rho((\mathcal{C}^{-1} \mathcal{C}', 0 ))\ket{x_i} = \ket{x_i}.
    \label{eqlemt}
\end{align}
\label{lemTest}
\end{lemma}
\begin{proof}
 It follows from a direct use of \eqref{switch} in definition \eqref{eigenH}.
\end{proof}
 \noindent In other words, the equivalence classes $[\mathcal{C}]$ are the left cosets of the stabilizer subgroup $GL(D,q)_{x_i}$ of $x_i$ in $GL(D,q)$ and are independent of $S$. We thus find:
\begin{proposition}
 The set of vectors
 \begin{align}
     \{ \ket{[\mathcal{C}], S}_i : [\mathcal{C}] \in GL(D,q)/GL(D,q)_{x_i}, \ \mathcal{C}^t S \mathcal{C} \in \text{Sym}_{D,i}\}
     \label{thebas}
 \end{align}
 gives an orthonormal basis for the vector space $V_i$ which diagonalizes the matrices $\rho((\mathds{1},\mathcal{F}))$.
\end{proposition}
\noindent Using the basis of eigenvectors $\ket{[\mathcal{C}], S}_i$, we can write
\begin{equation}
    \begin{split}
         V_i &= \text{span}_{\mathbb{C}}\{ \ket{[\mathcal{C}], S}_i : [\mathcal{C}] \in GL(D,q)/GL(D,q)_{x_i}, \ \mathcal{C}^t S \mathcal{C} \in \text{Sym}_{D,i}\}\\
         &= \bigoplus_{S \in \text{Sym}_D} W_i(S),
    \end{split}
    \label{holabra}
\end{equation}
where $W_i(S)$ is defined as
\begin{align}
    W_i(S) = \text{span}_{\mathbb{C}}\{\ket{[\mathcal{C}], S}_i : \mathcal{C} \text{ s.t. } \mathcal{C}^t S \mathcal{C} \in \text{Sym}_{D,i}\}.
\end{align}
By construction, we have 
\begin{align}
    \rho((\mathds{1},\mathcal{F})) W_i(S) = e^{-\frac{2 \pi i}{q} \text{tr}(S \mathcal{F})} W_i(S).
\end{align}
These modules $W_i(S)$ correspond to the eigenspaces of $H$ contained in the $i^{\text{th}}$ neighborhood of $x_0$. The total eigenspaces $W(S)$ are defined as
\begin{align}
    W(S) = \bigoplus_{i = 0}^D W_i(S).
\end{align}
From \eqref{decrom} and \eqref{holabra}, we see that they give a decomposition of $V$ in eigenspaces of $H$:
 \begin{align}
     V = \bigoplus_{S \in \text{Sym}_D} W(S),
 \end{align}
 where
\begin{align}
    \rho((\mathds{1},\mathcal{F})) W(S) = e^{-\frac{2 \pi i}{q} \text{tr}(S \mathcal{F})} W(S).
\end{align}

\section{The standard module of the subspace lattice}
\label{sbonus}
 
Our next goal is to show that there is a one-to-one correspondence between the basis vectors $\ket{[\mathcal{C}],S}_i$ of a submodule $W(S)$ and the characteristic vectors of a \textit{subspace lattice} $L_N(q)$. 

Subspace lattices $L_N(q)$ will be discussed in some length in section \ref{s6}. At this point, it will suffice to say that they are graphs with vertices (and therefore characteristic vectors) labeled by the subspaces $\mathbb{V}$ of a $N$-dimensional vector space $\mathbb{V}_0$ over $\mathbb{F}_q$. The standard module $V_{L_N(q)}$ of the subpace lattice $L_N(q)$ is given by
\begin{align}
    V_{L_N(q)} = \text{span}_{\mathbb{C}}\{\ket{\mathbb{V}} : \mathbb{V} \subseteq \mathbb{V}_0 \},
\end{align}
where $\mathbb{V}_0 \cong \mathbb{F}_q^N$ as a vector space.

Recall that $\ket{[\mathds{1}],S}_i$ is constructed from vectors associated to subspaces sharing the same intersection with $x_0$ as $x_i$, i.e. $x_i \cap x_0$. Since the group elements $(\mathcal{C},0) \in P$ preserve $x_0$, $\ket{[\mathcal{C}],S}_i$ is constructed from vectors associated to subspaces $x$ of intersection $(\mathcal{C},0) x_i \cap x_0$ with $x_0$. In other words, for $x \in X$, we find
\begin{align}
\bra{x}\ket{[\mathcal{C}],S}_i = 0,
\end{align}
whenever
\begin{align}
    x \cap x_0 \neq (\mathcal{C},0) x_i \cap x_0.
\end{align}
Consider the map $\phi$ from $W(S)$ to $V_{L_D(q)}$ which acts as
\begin{align}
     \phi\left(\ket{[\mathcal{C}], S}_i\right) = \ket{ (\mathcal{C},0)x_i \cap x_0}.
     \label{defphi}
\end{align}
Lemma \ref{lemTest} ensures that $\phi$ is injective. The basis vectors $\ket{[\mathcal{C}],S}_i$ are thus in one-to-one correspondence with their image in $V_{L_D(q)}$. The identification of $\phi(W(S))$ is the purpose of the following proposition and the central result of this section.
\begin{proposition}
 Take $S \in \text{Sym}_D$ and let $N = \text{dim}(\text{ker}(S))$. The map $\phi$ defined in \eqref{defphi} is an isomorphism between $W(S)$ and $V_{L_N(q)}$.
 \label{propopo}
\end{proposition} 
\noindent The proof of this proposition requires a few lemmas.

\begin{lemma}
For any $S \in \text{Sym}_D$ and $\mathcal{C} \in GL(D,q)$, we have
\begin{align}
    \rho((\mathcal{C},0)) W(S) = W((\mathcal{C}^{-1})^t S \mathcal{C}^{-1}).
\end{align}
\end{lemma}
\begin{proof}
 Let $\ket{[\mathcal{C}'],S }_i$ be a basis vector of $W(S)$. We find that
\begin{equation}
    \begin{split}
        \rho((\mathcal{C},0)) \ket{[\mathcal{C}'],S }_{i} &= \rho((\mathcal{C},0)\cdot(\mathcal{C}',0))\ket{[\mathds{1}], {\mathcal{C}'}^t S \mathcal{C}'}_{i}, \\
        &=  \rho((\mathcal{C}\mathcal{C}',0))\ket{[\mathds{1}], {\mathcal{C}'}^t\mathcal{C}^t(\mathcal{C}^{-1})^t S \mathcal{C}^{-1}\mathcal{C}\mathcal{C}'}_{i},\\
        & = \ket{[\mathcal{C}\mathcal{C'}],(\mathcal{C}^{-1})^t S \mathcal{C}^{-1} }_{i}.
    \end{split}
\end{equation}
Thus, we get that $\rho((\mathcal{C},0)) W(S) \subseteq W((\mathcal{C}^{-1})^t S \mathcal{C}^{-1})$. Since $\mathcal{C}$ is invertible, we find $\rho((\mathcal{C},0)) W(S) \supseteq W((\mathcal{C}^{-1})^t S \mathcal{C}^{-1})$ and the lemma follows. 
\end{proof}

In the following, we denote $(M)_{{a \leq m \leq b, \ c \leq n \leq D }}$ the submatrix of a matrix $M$ formed by its rows $m$   and its columns $n$ verifying $ a \leq m  \leq b$ and $c \leq n  \leq D$.

\begin{lemma}
Given a module $W(S)$, there exists a matrix $\mathcal{C}\in GL(D,q)$ and a matrix $S' \in \text{Sym}_{D,i}$ for some $i \in \{0,1,\dots, D\}$, such that the non-zero submatrix $(S')_{D-i < m,n \leq D}$ of $S'$ verifies
\begin{align}
   \text{dim}(\text{ker}((S')_{ D-i < m,n \leq D} )) = 0
   \label{lem11}
\end{align}
and such that
\begin{align}
   \rho((\mathcal{C},0)) W(S) = W(S').
   \label{lem12}
\end{align}
\label{lemmod}
\end{lemma}

\begin{proof}
 Let $i_0$ be the smallest integer in $\{0,1,2,\dots,D\}$ such that 
\begin{align}
   \text{dim}(W_{i_0}(S)) \neq 0.
   \label{assum}
\end{align}
\noindent By construction, there exists a vector  $ \ket{\psi}$ in $W_{i_0}(S)$ and an element $(\mathcal{C}_1,0) \in P$ verifying
\begin{align}
    \ket{\psi} = \rho((\mathcal{C}_1,0)) \ket{[\mathds{1}],\mathcal{C}_1^t S \mathcal{C}_1}_{i_0},
\end{align}
\noindent with $\mathcal{C}_1^t S \mathcal{C}_1 \in \text{Sym}_{D,i_0}$. Now, let us assume that 
\begin{align}
    \text{dim}(\text{ker}((\mathcal{C}_1^t S \mathcal{C}_1)_{D-i_0 < m,n \leq D} )) \neq 0.
    \label{kercon}
\end{align}
\noindent We can then use any vector $v \in \mathbb{F}_q^{i_0}$ in the kernel to construct the following matrix $\mathcal{C}_2$:
\begin{align}
    \mathcal{C}_2 = 
    \begin{pmatrix}
\mathds{1}_{D-i_0 \times D-i_0} &  0\\
0 & (\mathcal{C}_2)_{D-i_0 < m,n \leq D}
\end{pmatrix},
\end{align}
\noindent where $(\mathcal{C}_2)_{D-i < m,n \leq D}$ is non-singular and has $v$ as its first column. One can check that $\mathcal{C}_2^t\mathcal{C}_1^t S \mathcal{C}_1\mathcal{C}_2 \in \text{Sym}_{D,i_0 -1}$ and thus 
\begin{align}
    \rho((\mathcal{C}_1\mathcal{C}_2, 0)) \ket{[\mathds{1}],(\mathcal{C}_1\mathcal{C}_2)^tS\mathcal{C}_1\mathcal{C}_2}_{i_0 - 1} = \ket{[\mathcal{C}_1\mathcal{C}_2],S}_{i_0 - 1} \in W_{i_0 - 1}(S).
\end{align}
Therefore, we get that $W_{i_0-1}(S)$ is not empty which contradicts the assumption that $i_0$ is the smallest integer such that \eqref{assum} is verified. We conclude that the kernel in \eqref{kercon} contains only the vector $0$ and we find \eqref{lem12} by applying lemma \ref{lemmod} with $\mathcal{C} = \mathcal{C}_1^{-1}$.
\end{proof}

\begin{lemma}
Let $i_0$ be the smallest integer such that $\text{dim}(W_{i_0}(S)) \neq 0$. Then,
\begin{align}
    \text{dim}(W_{i_0}(S)) = 1.
\end{align}
\end{lemma}
\begin{proof}
 By the previous lemma, we only need to check this for the modules $W(S)$ associated to matrices $S \in \text{Sym}_{D,i_0}$ verifying
\begin{align}
    \text{dim}(\text{ker}((S)_{D-i_0 < m,n \leq D} )) = 0.
        \label{kerker}
\end{align}
\noindent The results for the other modules $W(S)$ are obtained using maps $(\mathcal{C},0)$ in $P$. Since $S \in \text{Sym}_{D,i_0}$, we get that
\begin{align}
    \ket{[\mathds{1}],S}_{i_0} \in W_{i_0}(S).
\end{align}
\noindent Let us assume that there is another basis vector  $\ket{[\mathcal{C}], S}_{i_0}$ in $W_{i_0}(S)$. By construction, we have 
\begin{align}
  \ket{[\mathcal{C}], S}_{i_0} = \rho((\mathcal{C},0)) \ket{[\mathds{1}], S'}_{i_0} \in W_{i_0}(S),
   \label{sec}
\end{align}
\noindent with $S' = \mathcal{C}^t S \mathcal{C} \in \text{Sym}_{D, i_0}$. Since $\mathcal{C}$ is non-singular, \eqref{kerker} implies that
\begin{align}
    \text{dim}(\text{ker}((S')_{D-i_0 < m,n \leq D}) ) = 0.
    \label{sec2}
\end{align}
Equations \eqref{sec}, \eqref{sec2} and the fact that $S,S' \in \text{Sym}_{D,i_0}$ impose restrictions on the nature of $\mathcal{C}$. Indeed, one finds that
\begin{align}
    (\mathcal{C})_{{D-i_0<m \leq D, \  1 \leq n \leq D-i_0}} = 0.
\end{align}
\noindent Yet, this condition ensures that
\begin{align}
    \rho((\mathcal{C},0))\ket{x_i} = \ket{x_i}.
    \label{stabilili}
\end{align}
Applying lemma \ref{lemTest}, one gets 
\begin{align}
    \rho((\mathcal{C},0)) \ket{[\mathds{1}], S'}_{i_0}  =  \ket{[\mathds{1}], S}_{i_0},
\end{align}
which implies that
\begin{align}
    \ket{[\mathcal{C}], S}_{i_0} = \ket{[\mathds{1}], S}_{i_0}
\end{align}
and that $W_{i_0}(S)$ is one-dimensional. 
\end{proof}
\begin{lemma}
Let $i_0$ be the smallest integer such that $\text{dim}(W_{i_0}(S)) \neq 0$. Let $\ket{[\mathcal{C}],S}_{i_0}$ be the basis vector of the one-dimensional $W_{i_0}(S)$ and  $\ket{[\mathcal{C}'],S}_{i}$ be any basis vector of $W_i(S)$, $i \geq i_0$. If $\mathbb{V}_0$ and $\mathbb{V}$ are the intersections with $x_0$ shared by all the subspaces $x$ such that $\bra{x}\ket{[\mathcal{C}],S}_{i_0} \neq 0$ and $\bra{x}\ket{[\mathcal{C}'],S}_{i} \neq 0$ respectively, then
\begin{align}
    \mathbb{V} \subseteq \mathbb{V}_0.
\end{align}
\label{masterlem}
\end{lemma}
\begin{proof}
Again, we assume that $S \in \text{Sym}_{D, i_0}$ and $\text{dim}(\text{ker}((S)_{ D-i_0< m,n \leq D} ) )= 0$. The basis vector of the one-dimensional $W_{i_0}(S)$ is thus $\ket{[\mathds{1}],S}_{i_0}$ and we have
 \begin{align}
     \mathbb{V}_0 = x_{i_0} \cap x_0.
 \end{align}
 Now, consider
 \begin{align}
     \ket{[\mathcal{C}'],S}_i = \rho((\mathcal{C}',0))\ket{[\mathds{1}],S'}_i,
 \end{align}
 where $S' = (\mathcal{C}')^t S \mathcal{C}' \in \text{Sym}_{D,i}$. Note that $\mathcal{C}'$ maps the kernel of the matrix $S'$ to the kernel of $S$: 
 \begin{align}
     \mathcal{C}'\  \text{ker}(S') = \text{ker}(S).
     \label{coker}
 \end{align}
 We know that the kernel of $S$ is composed of the column vectors of dimension $D$ with zeros in their last $i_0$ entries. The condition $S' \in \text{Sym}_{D,i}$ also guarantees that the kernel of $S'$ contains the column vectors of dimension $D$ with zeros in their last $i$ entries. These observations and equation \eqref{coker} are sufficient to show that
 \begin{align}
     (\mathcal{C}',0) x_i \cap x_0 \subseteq x_{i_0}\cap x_0.
 \end{align}
 Since $(\mathcal{C}',0) x_i \cap x_0$ corresponds to $\mathbb{V}$, i.e to the intersection with $x_0$ of the subspaces $x$ verifying $\bra{x}\ket{[\mathcal{C}'],S}_i \neq 0$, we find
 \begin{align}
     \mathbb{V} \subseteq \mathbb{V}_0.
     \label{inclu}
 \end{align}
For modules $W(S)$ with $S  \notin \text{Sym}_{D, i_0}$, one can use the transformation $(\mathcal{C},0)$ of lemma \ref{lemmod} to recover a module for which the assumption is valid. Applying $(\mathcal{C},0) \in P$ to a basis vector of $W(S)$ changes its intersections $\mathbb{V}_0$ and $\mathbb{V}$ to $(\mathcal{C},0) \mathbb{V}_0$ and $(\mathcal{C},0) \mathbb{V}$. Since this transformation is invertible, the inclusion in \eqref{inclu} is preserved.
\end{proof}
\noindent We are now ready to give a proof of proposition \ref{propopo}.
\begin{proof}{(Proposition \ref{propopo})}
 From lemma \ref{masterlem}, we see that $\phi$ maps the vectors $\ket{[\mathcal{C}],S}_i$ to vectors associated to subspaces of $\mathbb{V}_0$. In other words,
 \begin{align}
     \phi(\ket{[\mathcal{C}],S}_i) \in \{\ket{\mathbb{V}} : \mathbb{V} \subseteq \mathbb{V}_0\}.
 \end{align}
 It remains to show that for any $\mathbb{V} \subseteq \mathbb{V}_0$, there exists a vector $\ket{[\mathcal{C}],S}_i$ constructed from vectors associated to subspaces having an intersection $\mathbb{V}$ with $x_0$. If $\ket{[\mathcal{C}_0],S}_{i_0}$ is the basis vector of the one-dimensional $W_{i_0}(S)$, we have
\begin{align}
    \mathbb{V}_0 = (\mathcal{C}_0,0) x_{i_0}\cap x_0,
\end{align}
with $\mathcal{C}_0^t S \mathcal{C}_0 \in \text{Sym}_{D,i_0}$. Note that $\text{dim}(\mathbb{V}_0) = D-i_0 = \text{dim}(\text{ker}(S))$. We can write any $\mathbb{V} \subseteq \mathbb{V}_0$ of dimension $D-i$ as
\begin{align}
    \mathbb{V} = (\mathcal{C}_0,0)\cdot (\mathcal{C},0) x_{i}\cap x_0,
\end{align}
with $(\mathcal{C},0)x_{i_0} = x_{i_0}$. In particular, we can choose a matrix $\mathcal{C}$ of the form
 
\begin{align}
    \mathcal{C} = 
    \begin{pmatrix}
(\mathcal{C})_{1 \leq m,n\leq D-i_0} &  0\\
0 & \mathds{1}_{i_0 \times i_0}
\end{pmatrix},
\end{align}
so that $\mathcal{C}^t \mathcal{C}_0^t S \mathcal{C}_0\mathcal{C} = \mathcal{C}_0^t S \mathcal{C}_0$. We thus have
\begin{align}
    \ket{[\mathcal{C}_0\mathcal{C}], S}_i \in W(S)
\end{align}
and by construction
\begin{align}
    \phi(\ket{[\mathcal{C}_0\mathcal{C}], S}_i) = \ket{ \mathbb{V}}.
\end{align}
\end{proof}
The correspondence between the vectors $\ket{[\mathcal{C}],S}_i$ and the vectors $\ket{\mathbb{V}}$ provided by $\phi$ indicates the existence of a connection between the symplectic dual polar graphs and the \textit{subspace lattices}. In the next section, we present the latter and its relation with $U_{\sqrt{q}}(sl_2)$.

\section{Subspace lattices $L_N(q)$ and $U_{\sqrt{q}}(sl_2)$}
\label{s6}
In this section, we review features of the subspace lattice $L_N(q)$ \cite{TERWILLIGER2003463, watanabe2017algebra, terwilliger1990incidence}. This graph has for vertices $X$ the set of subspaces of $\mathbb{F}_q^{N}$. Two vertices (or subspaces) $\mathbb{V}$ and $\mathbb{V}'$ are connected by an edge when $\mathbb{V}$ covers $\mathbb{V}'$, i.e.
\begin{align}
    \mathbb{V}' \subset \mathbb{V} \quad \text{and} \quad |\text{dim}(\mathbb{V}) - \text{dim}(\mathbb{V}')| = 1,
\end{align}
or $\mathbb{V}'$ covers $\mathbb{V}$. As we saw in the previous section, the standard module $V_{L_N(q)}$ of this graph is defined as
\begin{align}
    V_{L_N(q)} = \text{span}_{\mathbb{C}}\{\ket{\mathbb{V}} : \mathbb{V} \subseteq \mathbb{F}_q^N\},
\end{align}
with $\ket{\mathbb{V}}$ the characteristic vectors associated to the vertices of $L_N(q)$. The adjacency matrix is the sum of a lowering matrix $L$ and a raising matrix $R$ acting on $V_{L_N(q)}$, whose entries in the basis $\{\ket{\mathbb{V}}: \mathbb{V} \subseteq \mathbb{F}_q^N\}$ are
\begin{align}
\bra{\mathbb{V}}L\ket{\mathbb{V}'} = 
\left\{
    \begin{array}{ll}
    	1  & \mbox{if } \mathbb{V} \subset \mathbb{V}', \ |\text{dim}(\mathbb{V}) - \text{dim}(\mathbb{V}')| = 1 \\
    	0 & \mbox{otherwise, }
    \end{array}
\right.
\label{L}
\end{align}
\begin{align}
\bra{\mathbb{V}}R\ket{\mathbb{V}'} = 
\left\{
    \begin{array}{ll}
    	1  & \mbox{if } \mathbb{V}' \subset \mathbb{V}, \ |\text{dim}(\mathbb{V}) - \text{dim}(\mathbb{V}')| = 1 \\
    	0 & \mbox{otherwise. }
    \end{array}
\right.
\label{R}
\end{align}
\noindent This graph is not distance-regular. Nevertheless, one can define a matrix $K$ acting diagonally on $\{\ket{\mathbb{V}}: \mathbb{V} \subseteq \mathbb{F}_q^N\}$ and playing a role similar to a dual adjacency matrix. The non-zero entries of $K$ are
\begin{align}
    \bra{\mathbb{V}}K\ket{\mathbb{V}} = q^{\frac{N}{2} - \text{dim}(\mathbb{V})}.
\label{K}
\end{align}
\noindent In \cite{TERWILLIGER2003463}, it is shown that $\hat{L}' = q^{\frac{1-N}{2}}L$, $\hat{R}' = R$ and $K$ give a representation of $U_{\sqrt{q}}(sl_2)$:
\begin{align}
    K \hat{L}' = q \hat{L}'K, \quad K\hat{R}' = q^{-1}\hat{R}'K, \quad [\hat{L}',\hat{R}'] = \frac{K - K^{-1}}{q^{\frac{1}{2}} -q^{-\frac{1}{2}}}.
\end{align}
\noindent The same is true for $\hat{L} = q^{\frac{1-N}{4}}L$, $\hat{R} = q^{\frac{1-N}{4}}R$ and $K$, since they differ from  $\hat{L}'$, $\hat{R}'$ and $K$ by an algebra automorphism \cite{klimyk2012quantum}. This realization as the advantage of verifying
\begin{align}
    \hat{R}^t = \hat{L}.
\end{align}
The representation given by $\hat{L}$, $\hat{R}$ and $K$ is fully reducible \cite{klimyk2012quantum}.  For an irreducible submodule $\mathcal{I}$ of dimension $\text{dim}(\mathcal{I}) = \ell + 1$, there exists a basis $\{\ket{n}\}_{0 \leq n \leq \ell}$ such that
\begin{align}
    \hat{R} \ket{n} = \sqrt{[n+1]_{\sqrt{q}}[\ell - n]_{\sqrt{q}}}\ket{n+1},
    \label{ra}
\end{align}
\begin{align}
    \hat{L} \ket{n} = \sqrt{[n]_{\sqrt{q}}[\ell - n + 1]_{\sqrt{q}}}\ket{n-1},
    \label{la}
\end{align}
\begin{align}
    K \ket{n} = q^{\frac{\ell -2n}{2}}\ket{n},
    \label{ka}
\end{align}
where 
\begin{align}
    [n]_{\sqrt{q}} = \frac{q^{\frac{n}{2} } - q^{-\frac{n}{2} }}{q^{\frac{1}{2} } - q^{-\frac{1}{2} }}.
\end{align}
The number of irreducible submodules of dimension $\ell + 1$ in $V_{L_N(q)}$ was found in \cite{watanabe2017algebra}. Indeed, let $\mathbb{E}_i^*$ denote the projector onto the $i^{\text{th}}$ eigenspace of $K$, i.e. onto the vertices associated to the subspaces of dimension $i$:
\begin{align}
    \mathbb{E}_i^* = \sum_{\substack{\mathbb{V} \subseteq \mathbb{F}_q^N \\ \text{dim}(\mathbb{V}) = i}} \ket{\mathbb{V}}\bra{\mathbb{V}}.
\end{align}
\noindent Then, let the \textit{endpoint} $\nu$ of an irreducible representation $\mathcal{I}$ of $U_{\sqrt{q}}(sl_2)$ be the smallest integer such that $\mathbb{E}_{\nu}^* \mathcal{I} \neq 0$. It was found in \cite{watanabe2017algebra} that the number $\text{mult}({\nu},N)$ of irreducible $U_{\sqrt{q}}(sl_2)$-submodules in $V_{L_N(q)}$ having the endpoint $\nu$ is given by
\begin{align}
    \text{mult}({\nu},N) = 
    \left\{
    \begin{array}{ll}
        1 & \mbox{if } \nu = 0, \\
    	 \binom{N}{\nu}_q - \binom{N}{\nu - 1}_q  & \mbox{if } 0 < \nu \leq N/2, \\
    	0 & \mbox{otherwise, }
    \end{array}
\right.
    \label{dego}
\end{align}
where
\begin{align}
    \binom{N}{\nu}_q = \frac{(q^N - 1)(q^{N-1} - 1) \dots (q^{N-\nu + 1} - 1)}{(q^\nu - 1) (q^{\nu - 1} - 1) \dots (q - 1)},
\end{align}
\noindent and that $\nu$ is related to the dimension $\ell + 1$ of $\mathcal{I}$ by
\begin{align}
    \ell = N - 2\nu.
\end{align}
Given proposition \ref{propopo}, the decomposition of $V_{L_N(q)}$ in submodules $\mathcal{I}$ also yields a decomposition of $W(S)$. In the next section, we shall show that it corresponds to its decomposition in irreducible $\mathcal{T}$-submodules.

\section{The irreducible $\mathcal{T}$-submodules}
\label{s7}
In this section, we first identify the irreducible $\mathcal{T}$-submodules in $W(S)$ and provide their multiplicities. These new results are then related to the characterization of the submodules in terms of the endpoints and the diameter that was used in \cite{WORAWANNOTAI2013443}.

\subsection{The irreducible decomposition of $W(S)$}
The standard module $V$ of a symplectic dual polar graph is both a $\mathcal{T}$- and a $P$-module. Recall that for any $g \in P$, we have
\begin{align}
    [A, \rho(g)]  = [A^*, \rho(g)] = 0.
\end{align}
Therefore, one finds that
\begin{align}
    A W(S) \subseteq W(S) \quad \text{and} \quad A^* W(S) \subseteq W(S).
\end{align}
In particular, the eigenspaces $W(S)$ of the $H$-action on $V$ are $\mathcal{T}$-submodules. While they are not irreducible, we have a theorem which describes the restriction of $A$ and $A^*$ to $W(S)$. Let us define the \textit{type} $\epsilon$ of $S$ as
 \begin{align}
    \epsilon  =
    \left\{
    \begin{array}{ll}
    	1  & \mbox{if } \text{rank}(S) \text{ is $0$ \textbf{or} is even and  }(-1)^{\frac{\text{rank}(S)}{2}} \text{det}({Q}) \text{ is a square in }\mathbb{F}_q,\\
    	-1  & \mbox{if } \text{rank}(S) \text{ is even and }(-1)^{\frac{\text{rank}(S)}{2}} \text{det}({Q}) \text{ is a non-square in }\mathbb{F}_q, \\
    	0  & \mbox{if }\text{rank}(S)\text{ is odd,}
    \end{array}
\right.
\label{yo}
\end{align}
where $Q$ is a non-singular $\text{rank}(S) \times \text{rank}(S)$ matrix with entries in $\mathbb{F}_q$ verifying
 \begin{align}
    \Upsilon^t {S} \Upsilon = 
 \begin{pmatrix}
0 &  0\\
0 & Q
\end{pmatrix},
\end{align}
for some matrix $\Upsilon \in GL(D,q)$. We have
\begin{theorem}
Let $A$ and $A^*$ be the adjacency and dual adjacency matrix of a dual polar graph of type $[C_D(q)]$ with $q$ prime. Let $W(S)$ be an eigenspace of the abelian automorphism subgroup $H$ defined by \eqref{definitionH}. Then,
\begin{align}
    A|_{W(S)} = \epsilon q^{\frac{D}{2}}\mathcal{K}  - 1  + q^{D/2}\mathcal{K}^{1/2}(q^{-1/4}\hat{\mathcal{L}} + q^{1/4}\hat{\mathcal{R}} )
    \label{bigrez}
\end{align}
and
\begin{align}
    A^*|_{W(S)} = -\frac{q(q^{D-1}+1)}{q-1} + \frac{q(q^{D-1}+1)(q^D + 1)}{2(q-1)} q^{-D+N/2}\mathcal{K}^{-1},
    \label{medrez}
\end{align}
where $N = \text{dim}(\text{ker}(S))$, $\epsilon \in  \{-1,0,+1\}$  is the type of $S$ and 
\begin{align}
    \mathcal{K} =\phi^{-1} \circ K\circ \phi, \quad \mathcal{\hat{L}} = \phi^{-1} \circ\hat{L}\circ \phi \quad \text{and} \quad \mathcal{\hat{R}} = \phi^{-1} \circ\hat{R}\circ \phi
\end{align}
are up to an isomorphism $\phi$ defined by \eqref{defphi} the matrices given by \eqref{L}-\eqref{K} in the adjacency algebra of the subspace lattice $L_N(q)$.
\label{toto}
\end{theorem}
\begin{proof}
 It is obtained by the computation of the entries
 \begin{align}
     {}_i\bra{[\mathcal{C}],S} A \ket{[\mathcal{C}'],S}_j  \quad \text{and} \quad  {}_i\bra{[\mathcal{C}],S} A^* \ket{[\mathcal{C}'],S}_j,
 \end{align}
 which is presented in Appendix A. 
\end{proof}
This theorem clarifies the relation between the dual polar graphs and the subspace lattices, both considered as $q$-analogues of hypercubes. Indeed, equation \eqref{bigrez} expresses \textit{the restriction to $W(S)$ of the adjacency matrix of a symplectic dual polar graph as the adjacency matrix of a subspace lattice with loops and weights on its edges}. Furthermore, it yields the irreducible $\mathcal{T}$-submodules we were looking for. 

\begin{corollary}
The irreducible $\mathcal{T}$-submodules in $W(S)$ are given by $\phi^{-1}(\mathcal{I})$, with $\mathcal{I}$ an irreducible $U_{\sqrt{q}}(sl_2)$-submodule of $V_{L_N(q)}$.
\end{corollary}
\begin{proof}
 Using \eqref{ra}, \eqref{la} and \eqref{ka}, it is easy to find the action of $A|_{W(S)}$ and $A^*|_{W(S)}$ on $\phi^{-1}(\mathcal{I})$ and deduce its irreducibility as a $\mathcal{T}$-submodule.
\end{proof}

 Let us note that $\phi^{-1}(\mathcal{I})$ and $\phi^{-1}(\mathcal{I}') $ are isomorphic as submodules if and only if they share the same dimension $\ell + 1$ and are associated to matrices $S \in \text{Sym}_{D}$ having the same type $\epsilon$ and the same $\text{rank}(S) = D - N$, with $N = \text{dim}(\text{ker}(S))$. The classes of isomorphic submodules can thus be parameterized by the triplet $(\ell,\text{rank}(S), \epsilon)$ and we find the following result concerning the multiplicities:
 \begin{theorem}
 The number $M_{\ell, \text{rank}(S), \epsilon}$ of isomorphic $\mathcal{T}$-submodules of class $(\ell, \text{rank}(S), \epsilon)$ is
 \begin{align}
      M_{\ell, 0, 1} = \text{mult}({\frac{D - \ell}{2}},D),
     \label{odde1}
 \end{align}
 \begin{align}
     M_{\ell, 2n + 1, 0} = \text{mult}({\frac{D - 2n  -1 - \ell}{2}},D - {2n - 1})\prod_{k=1}^n \frac{q^{2k}}{q^{2k} - 1} \prod_{k=0}^{2n} (q^{D-k} - 1),
     \label{odde}
 \end{align}
  \begin{align}
     M_{\ell, 2n , \pm 1} =  
    	\text{mult}({\frac{D - 2n - \ell}{2}},D-{2n}) \frac{q^n \pm 1}{2q^n} \prod_{k=1}^n \frac{q^{2k}}{q^{2k} - 1} \prod_{k=0}^{2n-1} (q^{D-k} - 1),
\label{evene1}
 \end{align}
  with $\text{mult}(\nu,N)$ as defined in equation \eqref{dego}.
  \label{tototo}
 \end{theorem}
 \begin{proof}
 The multiplicity of the class $(\ell, D-N,\epsilon)$ is the product of the number of symmetric $D\times D$ matrices $S$ of rank $D-N$ and type $\epsilon$ and the number $\text{mult}({(N-\ell)/{2}},N)$ of irreducible $U_{\sqrt{q}}(sl_2)$-submodules $\mathcal{I}$ of dimension $\ell + 1$ in $V_{L_N(q)}$. The latter was given in section \ref{s6}. To identify the former, we recall equation \eqref{yo} which gives the type $\epsilon$ of $S$ in terms of $\text{rank}(S)$ and $\text{det}(Q)$. We are thus interested in the numbers $\mathcal{V}_+(D, D-N)$ (resp. $\mathcal{V}_-(D, D-N)$) of symmetric matrices of dimension $D$, of rank $D-N$ and for which the non-degenerate part $Q$ has a square (resp. a non-square) determinant. These numbers have been investigated in \cite{macwilliams1969orthogonal} and it was found that
\begin{align}
   \mathcal{V}_\pm(D, 2n+1) =  \frac{1}{2}\prod_{k=1}^n \frac{q^{2k}}{q^{2k} - 1} \prod_{k=0}^{2n} (q^{D-k} - 1)
\end{align}
\begin{align}\arraycolsep=1.4pt\def\arraystretch{1.8}
   \mathcal{V}_\pm(D, 2n) =  
      \left\{
    \begin{array}{ll}
     \frac{q^n \pm 1}{2q^n} \prod_{k=1}^n \frac{q^{2k}}{q^{2k} - 1} \prod_{k=0}^{2n-1} (q^{D-k} - 1)  & \mbox{if } -1 \text{ is a square in }\mathbb{F}_q \\
    	 \frac{q^n \pm (-1)^{{n}}}{2q^n} \prod_{k=1}^n \frac{q^{2k}}{q^{2k} - 1} \prod_{k=0}^{2n-1} (q^{D-k} - 1)  & \mbox{otherwise. }
    \end{array}
\right.
\end{align}
A direct application of these equations and of \eqref{yo} yields \eqref{odde} and \eqref{evene1}. To obtain \eqref{odde1}, we observe that $\text{rank}(S) = 0$ implies that $S = 0$ and $N = D$.
 \end{proof}

\subsection{Correspondence with the parameters $(r,t,d)$ }

In \cite{WORAWANNOTAI2013443}, the classes of isomorphic $\mathcal{T}$-submodules were parametrized by the triplet $(r,t,d)$. The endpoint $r$ is defined as the smallest integer such that $E_r^*\phi^{-1}(\mathcal{I}) \neq 0$. Similarly, the dual endpoint $t$ is defined as the smallest integer such that $E_t\phi^{-1}(\mathcal{I}) \neq 0$. As for $d$, it corresponds to the diameter of $\phi^{-1}(\mathcal{I})$, i.e. 
\begin{align}
    d = |\{ i \in \{0,1, \dots, D\} : E_i^* \phi^{-1}(\mathcal{I}) \neq 0 \}| - 1.
\end{align}
To find the correspondence between $(r,t,d)$ and $(\ell, \text{rank}(S), \epsilon)$, we use the known results \cite{WORAWANNOTAI2013443}:
\begin{align}
    A E_i = \theta_i E_i \quad \text{and} \quad A^* E^*_i = \theta^*_i E^*_i,
\end{align}
with,
\begin{align}
    \theta_i = \frac{1 - q + q^{D+1-i} - q^i}{q-1}
\end{align}
and
\begin{align}
    \theta_i^* = -\frac{q(q^{D-1}+1)}{q-1} + \frac{q(q^{D-1}+1)(q^D + 1))}{2(q-1)} q^{-i}.
\end{align}
One can use equations \eqref{medrez} and \eqref{ka} to check that $A^*$ acting on the basis $\{\phi^{-1}(\ket{n})\}_{0 \leq n \leq \ell }$ of $\phi^{-1}(\mathcal{I})$ yields
\begin{align}
    A^* \phi^{-1}(\ket{n}) = \theta_{D-n - \frac{N - \ell}{2}}\phi^{-1}(\ket{n}).
\end{align}
Therefore, we find for the submodule $\phi^{-1}(\mathcal{I})$ that
\begin{align}
    r = D - \frac{N + \ell}{2} \quad \text{and} \quad d = \ell.
    \label{corr1}
\end{align}
The same approach can be used to obtain the parameter $t$ in terms of $\ell$, $N$ and $\epsilon$. First, we diagonalize $A$ on $\phi^{-1}(\mathcal{I})$. Let $\phi^{-1}(\ket{\theta_k})$ be an eigenvector of $A$:
\begin{align}
    A \phi^{-1}(\ket{\theta_k}) = \theta_k\  \phi^{-1}(\ket{\theta_k}).
\end{align}
One obtains a three-term recurrence relation for $P_n(k) = \bra{\theta_k}\ket{n}$ by considering $\bra{\theta_k} \phi \circ A \circ \phi^{-1} \ket{n}$:
\begin{equation}
    \begin{split}
        \theta_k P_n(k) &= (\epsilon q^{\frac{D + \ell}{2} - n} - 1) P_n(k) + q^{\frac{D}{2} + \frac{\ell}{4} - \frac{n}{2} + \frac{1}{4}}\sqrt{[n]_{\sqrt{q}}[\ell - n + 1]_{\sqrt{q}}} P_{n-1}(k) \\ 
        & + q^{\frac{d}{2} + \frac{\ell}{4} - \frac{n}{2} - \frac{1}{4}}\sqrt{[n+1]_{\sqrt{q}}[\ell - n ]_{\sqrt{q}}} P_{n+1}(k) \\
        &= (\epsilon q^{\frac{D + \ell}{2} - n} - 1) P_n(k) + \frac{q^{D/2}}{(q^{1/2} -q^{-1/2})} \sqrt{(1-q^{-n})(q^{\ell-n+1} - 1)} P_{n-1}(k) \\ 
        & +  \frac{q^{D/2}}{(q^{1/2} -q^{-1/2})} \sqrt{(1-q^{-n-1})(q^{\ell-n} - 1)} P_{n+1}(k).
    \end{split}
\end{equation}
\noindent Using 
\begin{align}
    c =  
\left\{
    \begin{array}{ll}
    	-1  & \mbox{if } \epsilon = 0, \\
    	-1/q & \mbox{otherwise, }
    \end{array}
\right.
\end{align}
\noindent and comparing with the three-term recurrence relation of the dual $q$-Krawtchouk polynomials $K_n(q^{-x} + cq^{x-\ell}; c,\ell; \frac{1}{q})$ \cite{koekoek1996askey}, one finds for $x \in \{0, 1,\dots, \ell\}$

\begin{align}\arraycolsep=1.4pt\def\arraystretch{1.8}
    P_n(k) = \left\{
    \begin{array}{ll}
    \sqrt{\mathcal{W}(n,x)} K_{n}(q^{-x} + cq^{x-\ell}; c, \ell ; \frac{1}{q})  & \mbox{if } \epsilon = 0, \\
    (\epsilon)^n \sqrt{\mathcal{W}(n,x)} K_{n}(q^{-x} + cq^{x-\ell}; c, \ell ; \frac{1}{q}) & \mbox{otherwise, }
    \end{array}
\right. 
\end{align}
with
\begin{align}
    \mathcal{W}(n,x) = \frac{(cq^\ell, q^\ell, q^{-1})_x (1-cq^{\ell-2x})c^{-x} q^{x(x-2\ell)} (q^N ; q)_n}{(q,cq;q)_x (1-cq^\ell)(c^{-1};q^{-1})_\ell(q^{-1},q^{-1})_n (cq^\ell)^n},
\end{align}
and
\begin{align}
    \theta_k = 
    \left\{
    \begin{array}{ll}
      \frac{(q^{\frac{D-\ell + 1}{2} + x} - q^{\frac{D+\ell + 1}{2} -x})}{(q-1)} - 1& \mbox{if } \epsilon = 0, \\
   \frac{\epsilon(q^{\frac{D - \ell}{2} + x + 1} - q^{\frac{D + \ell}{2} - x })}{(q-1)} - 1 & \mbox{otherwise. }
    \end{array}
\right.
\end{align}
This leads to the identification
\begin{align}
k = 
        \left\{
    \begin{array}{ll}
       \frac{D + \ell + 1}{2} - x& \mbox{if } \epsilon = 0, \\
    \frac{D - \epsilon \ell + 1 - \epsilon }{2} + \epsilon x & \mbox{otherwise. }
    \end{array}
\right.
\end{align}
We can therefore conclude that
\begin{align}
t =  \frac{D - \ell + 1 - \epsilon }{2}.
\label{corr2}
\end{align}
The multiplicity for each class $(r,t,d)$ is easily obtained from theorem \ref{tototo} using equations \eqref{corr1} and \eqref{corr2}. 

\section{Concluding remarks}
\label{s8}
We have investigated the relation between two $q$-analogues of hypercubes: the dual polar graphs of type $[C_D(q)]$ and the subspace lattices $L_N(q)$. For $q$ prime, we have shown that the restriction of the adjacency matrix of a symplectic dual polar graph to an eigenspace $W(S)$ of the abelian automorphism subgroup $H$ corresponds to the adjacency matrix of a \textit{weighted} subspace lattice. Furthermore, we have used the connection between $L_N(q)$ and $U_{\sqrt{q}}(sl_2)$ to find the irreducible $\mathcal{T}$-submodules of the standard module $V$ of the symplectic dual polar scheme and their multiplicities.

We expect that similar results can be obtained when $q$ is a prime power. We still have in that case a combinatorial model for the graph and the abelian automorphism subgroup $H$ exists. However, $H$ is no longer isomorphic to multiple copies of a cyclic group and further work is required. 

Other types of dual polar graphs are obtained by replacing the symplectic form $\mathfrak{B}$ with a quadratic or hermitiean one. Therefore, they are related to other groups of Lie type, such as $O(2n + 1,q)$, $O^\pm(2n,q)$ and $U(2n,q^2)$. It should prove interesting to check if these distance-regular graphs can be studied using an approach similar to the one we applied here.

As an additional remark, let us note that theorem \ref{toto} and \ref{tototo} and equations \eqref{ra}, \eqref{la} and \eqref{ka} give a recipe for the adjacency and dual adjacency matrices which is independent of the combinatorial model of the graph. In particular, they offer a construction which does not use finite fields $\mathbb{F}_q$ or groups of Lie type. In fact, the matrices $A$ and $A^*$ are well defined this way as long as the multiplicities \eqref{odde1}, \eqref{odde} and \eqref{evene1} of submodules are integers. Interestingly, this does not seem to require $q$ to be a prime power. For instance, we find an expression for $A$ which corresponds to the adjacency matrix of an hypercube when taking $q \rightarrow 1$. It remains to explore whether this construction yields matrices $A$ and $A^*$ associated to distance-regular graphs for other values of $q$ which are not prime powers. 

Finally, it is our intent to use the results presented here to study the entanglement of free fermions on these graphs we have investigated. 

\section*{Acknowledgements}

The authors are grateful to Paul Terwilliger for comments. PAB holds a scholarship from the Natural Sciences and Engineering Research Council of Canada (NSERC). The research of LV is supported in part by a Discovery Grant from NSERC.

\bigskip

\noindent \textbf{Declarations of interest}: none.

\appendix
\section{proof of theorem \ref{toto}}

First, we fix $S$ and consider the computation of
\begin{align}
     A_{i,j,[\mathcal{C}],[\mathcal{C}']} = {}_i\bra{[\mathcal{C}],S} A \ket{[\mathcal{C}'],S}_j.
\end{align}
We can use lemma $3.9$ of \cite{thini} to identify the entries $A_{i,j,[\mathcal{C}],[\mathcal{C}']} \neq 0$. It gives
\begin{align}
    E_m^* A E_n^* = 0 \quad \text{if} \quad |m - n| >1,
\end{align}
where $E_i^*$ is the projector onto the $i^{\text{th}}$ neighborhood of $x_0$ defined in \eqref{defdualP}. The construction of the eigenvectors of $H$ also implies that
\begin{align}
    E_m^* \ket{[\mathcal{C}],S}_n = \delta_{mn}\ket{[\mathcal{C}],S}_n.
\end{align}
We thus deduce that $A_{i,j,[\mathcal{C}],[\mathcal{C}']}$ is only non-zero when $|i-j|\leq 1$. Another condition arises from the development
\begin{equation}
    \begin{split}
        A_{i,j,[\mathcal{C}],[\mathcal{C}']} & = \sum_{x,y \in X} {}_i\bra{[\mathcal{C}],S}\ket{x}\bra{x} A \ket{y}\bra{y}\ket{[\mathcal{C}'],S}_j.
        \label{decA}
    \end{split}
\end{equation}
Terms in the sum can only be non-zero when $x$ and $y$ have respectively the intersections $(\mathcal{C},0)x_i \cap x_0$ and $(\mathcal{C}',0)x_j \cap x_0$ with $x_0$. The coefficient $\bra{x}A\ket{y}$ also requires  $\text{dim}(x \cap y) = D - 1$. Then, from the following inequality
\begin{align}
    \text{dim}(x\cap y) \leq  D - (\text{max}(D-i,D-j)  - \text{dim}(x\cap y \cap x_0)),
\end{align}
one concludes that \eqref{decA} can only be non-zero in four cases:

\begin{itemize}
    \item[\textbf{(I)}] $i = j$ and $(\mathcal{C},0)x_i \cap x_0 =(\mathcal{C}',0)x_j \cap x_0$;
    \item[\textbf{(II)}] $i = j$ and $\text{dim}\left((\mathcal{C},0)x_i \cap (\mathcal{C}',0)x_j \cap x_0\right) = D-i -1$;
    \item[\textbf{(III)}] $i = j-1$ and $(\mathcal{C},0)x_i \cap x_0 $ \textit{covers} $ (\mathcal{C}',0)x_j \cap x_0$;
    \item[\textbf{(IV)}]$i = j+1$ and $(\mathcal{C},0)x_i \cap x_0 $ is \textit{covered} by $ (\mathcal{C}',0)x_j \cap x_0$.
\end{itemize}
We shall consider these cases one by one.

\bigskip

\noindent \textbf{Case (I)}: Consider $A_{i,i,[\mathcal{C}],[\mathcal{C}]}$. We can use the fact that $\rho(P)$ is in the centralizer of $\mathcal{T}$ to show that
\begin{equation}
    \begin{split}
        A_{i,i,[\mathcal{C}],[\mathcal{C}]} &=   {}_i\bra{[\mathds{1}],S'}\rho((\mathcal{C},0)^{-1}) A \rho((\mathcal{C},0))\ket{[\mathds{1}],S'}_i,\\
        &= {}_i\bra{[\mathds{1}],S'} A\ket{[\mathds{1}],S'}_i,
    \end{split}
\end{equation}
where $S' = \mathcal{C}^t S \mathcal{C}$. Expanding $\ket{[\mathds{1}],S'}_i$ with the help of definition \eqref{eigenH} yields
\begin{equation}
\begin{split}
    {}_i\bra{[\mathds{1}],S'} A\ket{[\mathds{1}],S'}_i &= \sum_{\mathcal{F},\mathcal{F}' \in \text{Sym}_{D,i}} \chi_{S'}^*(\mathcal{F})\chi_{S'}(\mathcal{F}') \bra{x_i}\rho((\mathds{1},\mathcal{F})^{-1})A\rho((\mathds{1},\mathcal{F}'))\ket{x_i},\\
    &= \sum_{\mathcal{F},\mathcal{F}' \in \text{Sym}_{D,i}} \sum_{\substack{x \in X : \\
    \text{dist}(x, x_i) = 1}} \chi_{S'}^*(\mathcal{F})\chi_{S'}(\mathcal{F}') \bra{x_i}\rho((\mathds{1},\mathcal{F}'- \mathcal{F})) \ket{x}.
     \end{split}
     \end{equation}
\noindent We have
\begin{align}
\bra{x_i}\rho((\mathds{1},\mathcal{F}'- \mathcal{F})) \ket{x} = 
    \left\{
    \begin{array}{ll}
    	1  & \mbox{if } x = (\mathds{1},\mathcal{F}'- \mathcal{F})^{-1} x_i, \\
    	0 & \mbox{otherwise. }
    \end{array}
\right.
\label{cocoef}
\end{align}
Since $\text{dim}(x\cap x_i) = D-1$, we deduce from the action of $H$ on $\mathbb{F}_q^{2D}$ that we can only get $ x = (\mathds{1},\mathcal{F}'- \mathcal{F})^{-1} x_i$ when $\text{rank}(\mathcal{F}' - \mathcal{F}) = 1$. Therefore,
\begin{equation}
\begin{split}
    {}_i\bra{[\mathds{1}],S'} A\ket{[\mathds{1}],S'}_i 
    &= \sum_{\substack{\mathcal{F},\mathcal{F}' \in \text{Sym}_{D,i}\\
    \text{rank}(\mathcal{F} - \mathcal{F}') = 1}}  \chi_{S'}^*(\mathcal{F})\chi_{S'}(\mathcal{F}'), 
    \\ &= \sum_{\substack{\mathcal{F},\mathcal{F}' \in \text{Sym}_{D,i}\\
    \text{rank}(\mathcal{F} - \mathcal{F}') = 1}}  q^{-\frac{i(i + 1)}{4}}\chi_{S'}(\mathcal{F}' - \mathcal{F}).
     \end{split}
     \end{equation}
Next, we define $\mathcal{F}_\pm = \mathcal{F} \pm \mathcal{F}' \in \text{Sym}_{D,i}$. The condition $\text{rank}(\mathcal{F}_-) = 1$ implies that the submatrices $(\mathcal{F}_-)_{ D-i< m,n \leq D}$ (i.e. the non-zero part of the matrices $\mathcal{F}_-$) are given by
\begin{align}
    \{ av v^t : a \in \mathbb{F}_q\backslash\{0\}, \ v \in P(\mathbb{F}_q^{i}) \},
\end{align}
where $P(\cdot)$ refers to the projective space. Let $\Bar{S}' = ({S}')_{D-i < m,n \leq D} $ be the non-zero block of $S'$. We find
     \begin{equation}
   \begin{split}
 {}_i\bra{[\mathds{1}],S'} A\ket{[\mathds{1}],S'}_i &= \sum_{ \mathcal{F}_+\in \text{Sym}_{D,i} } \sum_{ a \in \mathbb{F}_q\backslash\{0\}} \sum_{v \in P(\mathbb{F}_q^i)} q^{-\frac{i(i + 1)}{4}}\chi_{S'}(\mathcal{F}_-),\\
     & = \sum_{ a \in \mathbb{F}_q\backslash\{0\}} \sum_{v \in P(\mathbb{F}_q^i)}  e^{\frac{2 \pi i}{q} a \text{tr}(\Bar{S}' vv^t)},\\
     &= \sum_{ a \in \mathbb{F}_q\backslash\{0\}} \sum_{v \in P(\mathbb{F}_q^i)}  e^{\frac{2 \pi i}{q} a v^t\Bar{S}' v}.
     \end{split}
\end{equation}
Splitting the summation on $v \in P(\mathbb{F}_q^i)$ in the two cases $v^t (\Bar{S}')^tv = 0$ and $v^t (\Bar{S}')^tv \neq 0$, one finds
\begin{equation}
    \begin{split}
  {}_i\bra{[\mathds{1}],S'} A\ket{[\mathds{1}],S'}_i&= \sum_{\substack{v \in P(\mathbb{F}_q^{ i}) \\ v^t (\Bar{S}')^tv = 0 }} (q-1) + \sum_{\substack{v \in P(\mathbb{F}_q^{ i}) \\ v^t (\Bar{S}')^tv \neq 0 }}\sum_{\substack{a \in \mathbb{F}_q\backslash \{0\}}}e^{\frac{2 \pi i}{q}a v^t (\Bar{S}')^t v},\\
    &= \sum_{\substack{v \in P(\mathbb{F}_q^{ i}) \\ v^t (\Bar{S}')^tv = 0 }} (q-1) - \sum_{\substack{v \in P(\mathbb{F}_q^{ i}) \\ v^t (\Bar{S}')^tv \neq 0 }} 1.
\end{split}
\end{equation}
\noindent For $S=0$, this is easy to compute and we get
\begin{equation}
    \begin{split}
  {}_i\bra{[\mathds{1}],0} A\ket{[\mathds{1}],0}_i&= q^i - 1.
\end{split}
\end{equation}
\noindent For $S \neq 0$, it is a matter of determining how many $v \in P(\mathbb{F}_q^i)$ fall in the cases $v^t (\Bar{S}')^tv = 0$ and $v^t (\Bar{S}')^tv \neq 0$ respectively. To do so, we consider $\Bar{S}'$ as the matrix representation of a quadratic form over $\mathbb{F}_q^i$. For any such form, it is known \cite{schmidt2006equations} that there exists a non-singular $i \times i$ matrix $\Upsilon \in GL(i,q)$ such that
\begin{align}
    \Upsilon^t \Bar{S}' \Upsilon = 
 \begin{pmatrix}
0 &  0\\
0 & Q
\end{pmatrix},
\label{deftype1}
\end{align}
where $Q$ is a non-singular $k \times k $ matrix, $k = \text{rank}(\Bar{S}')$. In particular, $Q$ defines a non-degenerate quadratic form over $\mathbb{F}_q^k$. We can thus use the following result concerning non-degenerate quadratic forms \cite{ LEEP1999157,schmidt2006equations}:
\begin{align}
   |\{v \in \mathbb{F}_q^k : v^t Q v = 0 \}| = 
q^{k-1} + \epsilon (q-1)q^{\frac{k-2}{2}}
\label{Nnn}
\end{align}
where the \textit{type} $\epsilon$ of $Q$ is defined as 
\begin{align}
    \epsilon  =
    \left\{
    \begin{array}{ll}
    	1  & \mbox{if }  k \text{ is even and } (-1)^{\frac{n}{2}} \text{det}({Q}) \text{ is a square in }\mathbb{F}_q, \\
    	-1  & \mbox{if }  k \text{ is even and } (-1)^{\frac{n}{2}} \text{det}({Q}) \text{ is a non-square in }\mathbb{F}_q, \\
    	0  & \mbox{if } k \text{ is odd.}
    \end{array}
\right.
\end{align}
In the following, we also take  $\epsilon = 1$ when $S = 0$ to keep the notation compact.  From \eqref{Nnn}, we can deduce the number $\mathcal{N}_{i,S}$ of isotropic lines in the space $\mathbb{F}_q^i$ equipped with the quadratic form $\Bar{S}'$:
\begin{align}
    \mathcal{N}_{i,S} = \frac{q^{i-D+N}}{q-1}(q^{D-N-1} + \epsilon(q-1)q^{\frac{D-N}{2} - 1}) - \frac{1}{q-1},
\end{align}
\noindent where $N = \text{dim}(\text{ker}(S))$. From there, we recall that the number of lines in $\mathbb{F}_q^i$ is $\frac{q^i - 1}{q-1}$ and find
\begin{equation}
    \begin{split}
    {}_i\bra{[\mathds{1}],S'} A\ket{[\mathds{1}],S'}_i &= q \mathcal{N}_{i,S}- \frac{q^i - 1}{q-1}.
\end{split}
\end{equation}

\noindent \textbf{Case (II)}: We need the following lemma.

\begin{lemma}
If two subspaces $x$ and $y$ in $X$ are such that $\text{dim}(x \cap x_0) = \text{dim}( y \cap x_0) = D-i $ and $\text{dim}(x \cap y \cap x_0) = D- i -1$, then $\text{dist}(x,y) \geq 2$.
\end{lemma}
\begin{proof}
 Proving that $\text{dist}(x,y) \neq 0 $ is trivial. Now, assume that the conditions in the wording of the lemma are satisfied and that $\text{dist}(x,y) = 1$. Then 
\begin{align}
    x = \text{span}_{\mathbb{F}_q}\{u_1, \dots, u_{D-i-1}\} \oplus \text{span}_{\mathbb{F}_q}\{ u_{D-i}\} \oplus \text{span}_{\mathbb{F}_q}\{u_{D-i+1}, \dots, u_{D}\} 
\end{align}
\noindent and
\begin{align}
    y = \text{span}_{\mathbb{F}_q}\{u_1, \dots, u_{D-i-1}\} \oplus \text{span}_{\mathbb{F}_q}\{ u_{D-i}'\} \oplus \text{span}_{\mathbb{F}_q}\{u_{D-i+1}, \dots, u_{D}\},
\end{align}
\noindent with $u_j \in x_0$ for $j < D-i$ and $u_j \notin x_0$ for $j > D-i$. Since  $\text{dim}(x \cap x_0) = \text{dim}( y \cap x_0) = D-i $, $u_{D-i}$ and $u'_{D-i}$ are both in $x_0$. Thus $\mathfrak{B}(u_{D-i},u'_{D-i}) = 0$ and we find that
\begin{align}
    z = x \oplus \text{span}_{\mathbb{F}_q}\{ u_{D-i}'\} = y \oplus \text{span}_{\mathbb{F}_q}\{ u_{D-i}\}
\end{align}
\noindent is an isotropic subspace. This is a contradiction, since $x$ and $y$ are maximal and $x,y \subseteq z$. Since both $\text{dist}(x,y) = 0$ and $\text{dist}(x,y) = 1$ are not possible, the result follows.
\end{proof}
Using this lemma and the conditions of case (II), we find that vectors $\ket{x}$ and $\ket{y}$ such that $\bra{x}\ket{[\mathcal{C}],S}_i \neq 0$ and  $\bra{y}\ket{[\mathcal{C}'],S}_i \neq 0$ are associated to subspaces verifying
\begin{align}
    \text{dist}(x,y) \geq 2.
\end{align}
In particular, $\bra{x}A\ket{y} = 0$. Therefore,
\begin{align}
     {}_i\bra{[\mathcal{C}],S} A \ket{[\mathcal{C}'],S}_i = 0,
\end{align}
when $[\mathcal{C}] \neq [\mathcal{C}']$. 

\bigskip

\noindent \textbf{Case (III)}: We now focus on
\begin{align}
    A_{i,i+1,[\mathcal{C}],[\mathcal{C}']} =  {}_i\bra{[\mathcal{C}],S} A \ket{[\mathcal{C}'],S}_{i+1}
\end{align}
such that
\begin{align}
     (\mathcal{C}',0)x_{i+1} \cap x_0 \subset (\mathcal{C},0)x_i\cap x_0.
     \label{inck}
\end{align} By construction, we have
\begin{align}
    \rho((\mathcal{C},0)^{-1})\ket{[\mathcal{C}'],S}_{i+1} = \ket{[\mathcal{C}^{-1}\mathcal{C}'],S'}_{i+1}, \quad \rho((\mathcal{C},0)^{-1})\ket{[\mathcal{C}],S}_{i} = \ket{[\mathds{1}],S'}_{i},
\end{align}
where $S' = \mathcal{C}^t S \mathcal{C} \in \text{Sym}_{D,i}$ and $(\mathcal{C}^{-1}\mathcal{C}')^t S' (\mathcal{C}^{-1}\mathcal{C}') \in \text{Sym}_{D,{i+1}}$. From \eqref{inck}, we also note that
\begin{align}
    (\mathcal{C}^{-1}\mathcal{C}',0)x_{i+1}\cap x_0 \subset x_i \cap x_0 = \text{span}_{\mathbb{F}_q}\{e_1, \dots e_{D-i}\}.
    \label{ink2}
\end{align}
We can therefore construct a non-singular matrix
\begin{align}
    \mathcal{C}'' = 
    \begin{pmatrix}
({\mathcal{C}''})_{1 \leq m,n \leq D-i} &  0\\
0 & \mathds{1}_{i \times i }
\end{pmatrix},
\end{align}
verifying
\begin{align}
    ({\mathcal{C}''})_{1 \leq m \leq D-i-1, 1 \leq n \leq D-i} = ({\mathcal{C}'}^{-1}\mathcal{C})_{1 \leq m \leq D-i-1, 1 \leq n \leq D-i}.
\end{align}
The following relations are verified
\begin{align}
    \rho((\mathcal{C}'', 0)) \ket{[\mathcal{C}^{-1}\mathcal{C}'],S'}_{i+1} = \ket{[\mathds{1}],S'}_{i+1},
\end{align}
\begin{align}
    \rho((\mathcal{C}'', 0))\ket{[\mathds{1}],S'}_{i} = \ket{[\mathds{1}],S'}_{i},
\end{align}
and we see that
\begin{equation}
    \begin{split}
        A_{i,i+1,[\mathcal{C}],[\mathcal{C}']} &=  {}_i\bra{[\mathds{1}],S'}\rho((\mathcal{C},0))A \rho((\mathcal{C},0)^{-1})\ket{[\mathcal{C}^{-1}\mathcal{C}'],S'}_{i+1}, \\
        &=  {}_i\bra{[\mathds{1}],S'}A\ket{[\mathcal{C}^{-1}\mathcal{C}'],S'}_{i+1}, \\
         &=  {}_i\bra{[\mathds{1}],S'}A\rho((\mathcal{C}'', 0))^{-1}\rho((\mathcal{C}'', 0))\ket{[\mathcal{C}^{-1}\mathcal{C}'],S'}_{i+1}, \\
        & = {}_i\bra{[\mathds{1}],S'} A \ket{[\mathds{1}],S'}_{i+1}.
    \end{split}
\end{equation}
\noindent Expanding the right side, we find
\begin{equation}
\begin{split}
    {}_i\bra{[\mathds{1}],S'} A\ket{[\mathds{1}],S'}_{i+1} &= \sum_{\substack{\mathcal{F} \in \text{Sym}_{D,i}, \\
    \mathcal{F}' \in \text{Sym}_{D,i+1}}} \chi_{S'}^*(\mathcal{F})\chi_{S'}(\mathcal{F}') \bra{x_i}(\mathds{1},\mathcal{F})^{-1}A(\mathds{1},\mathcal{F}')\ket{x_{i+1}},\\
    &= \sum_{\substack{\mathcal{F} \in \text{Sym}_{D,i}, \\
    \mathcal{F}' \in \text{Sym}_{D,i+1}}} \sum_{\substack{x \in X : \\
    \text{dist}(x, x_{i + 1}) = 1}} \chi_{S'}^*(\mathcal{F})\chi_{S'}(\mathcal{F}') \bra{x_i}(\mathds{1},\mathcal{F}'- \mathcal{F}) \ket{x}.
\end{split}
\end{equation}
The terms in the sum are non-zero only when $(\mathds{1},\mathcal{F}' - \mathcal{F})x = x_i$. Given a pair $\mathcal{F}$ and $\mathcal{F}'$, one can check that this happens for some unique $x$ verifying $ \text{dist}(x, x_{i + 1}) = 1$ as long as
\begin{align}
    (\mathcal{F})_{ D-i< m,n \leq D} = (\mathcal{F}')_{ D-i< m,n \leq D}.
    \label{cii}
\end{align} 
In particular, there is a unique matrix $\mathcal{F} \in \text{Sym}_{D,i}$ verifying \eqref{cii} with respect to a given $\mathcal{F}'$. We denote this matrix $\mathcal{F}'_i$ and thus find
    \begin{equation}
    \begin{split}
    {}_i\bra{[\mathds{1}],S'} A\ket{[\mathds{1}],S'}_{i+1}  &= \sum_{\mathcal{F}' \in \text{Sym}_{D,i+1}} q^{-\frac{i(i+1)}{4}}q^{-\frac{(i+1)(i+2)}{4}}e^{\frac{2 \pi i}{q} \text{tr}(S'(\mathcal{F}' - \mathcal{F}'_i)) }.
    \end{split}
    \label{77}
    \end{equation}
Yet, the conditions $(\mathcal{F}' - \mathcal{F}'_i)_{ D-i< m,n \leq D} = 0$ and $S' \in \text{Sym}_{D,i}$ are sufficient to show that $\text{tr}(S'(\mathcal{F}' - \mathcal{F}'_i)) = 0$. Therefore, we get
 \begin{equation}
    \begin{split}
    {}_i\bra{[\mathds{1}],S'} A\ket{[\mathds{1}],S'}_{i+1}  &= \sum_{\mathcal{F}' \in \text{Sym}_{D,i+1}} q^{-\frac{i(i+1)}{4}}q^{-\frac{(i+1)(i+2)}{4}}, \\
    &= q^{\frac{(i+1)}{2}}.
    \end{split}
    \label{88}
\end{equation}
\bigskip
\noindent \textbf{Case (IV)}: We deduce that
\begin{equation}
    \begin{split}
         A_{i+1,i,[\mathcal{C}],[\mathcal{C}']} &=  {}_{i+1}\bra{[\mathcal{C}],S} A \ket{[\mathcal{C}'],S}_{i}, \\
         & = q^{\frac{i+1}{2}}.
    \end{split}
\end{equation}
from case (III) since $A$ is a symmetric matrix. To sum up, we found
\begin{align}
 {}_{i}\bra{[\mathcal{C}],S} A \ket{[\mathcal{C}'],S}_{j} = 
     \left\{
    \begin{array}{ll}
    	q \mathcal{N}_{i,S}- \frac{q^i - 1}{q-1}  & \mbox{if } {}_{i}\bra{[\mathcal{C}],S} \phi^{-1} \circ \mathds{1} \circ \phi \ket{[\mathcal{C}'],S}_{j} = 1, \\
    q^{\frac{i+1}{2}}  & \mbox{if } {}_{i}\bra{[\mathcal{C}],S} \phi^{-1} \circ L \circ \phi \ket{[\mathcal{C}'],S}_{j} = 1, \\
    q^{\frac{i}{2}}  & \mbox{if } {}_{i}\bra{[\mathcal{C}],S} \phi^{-1} \circ R \circ \phi \ket{[\mathcal{C}'],S}_{j} = 1, \\
    	0   & \mbox{otherwise. }
    \end{array}
\right.
\end{align}
Since these entries fit those of \eqref{bigrez}, the first part of the theorem is proved. Next, we consider $A^*|_{W(S)}$. It is known that \cite{WORAWANNOTAI2013443}
\begin{align}
    \bra{x} A^* \ket{y} = \delta_{x,y}\theta_{\text{dist}(x,x_0)}^*,
\end{align}
where
\begin{align}
   \theta_k^*  = \frac{-q(q^{D-1} +1)}{q-1} + \frac{q(q^{D-1} +q)(q^{D} + 1)}{2(q-1)}q^{-k}.
\end{align}
Therefore, we find
\begin{equation}
   \begin{split}
       {}_{i}\bra{[\mathcal{C}],S} A^* \ket{[\mathcal{C}'],S}_{j} &=  {}_{i}\bra{[\mathcal{C}],S} \sum_{k = 0 }^D \theta_k^* E_k^* \ket{[\mathcal{C}'],S}_{j}, \\
       &=\delta_{i,j}\delta_{[\mathcal{C}],[\mathcal{C}']}\theta^*_{i}
   \end{split} 
\end{equation}
and we can use
\begin{align}
     {}_{i}\bra{[\mathcal{C}],S} \phi^{-1} \circ K \circ \phi \ket{[\mathcal{C}'],S}_{j} =\delta_{i,j}\delta_{[\mathcal{C}],[\mathcal{C}']} q^{\frac{N}{2} - D + i}
\end{align}
to recover equation \eqref{medrez}.


\begin{thebibliography}{10}

\bibitem{Bannai1984AlgebraicCI}
E.~Bannai and T.~Ito.
\newblock {\em Algebraic {C}ombinatorics {I}: Association Schemes}.
\newblock Benjamin/Cummings, Menlo Park, 1984.

\bibitem{Hamming}
P.-A. Bernard, N.~Crampe, and L.~Vinet.
\newblock Entanglement of {F}ree {F}ermions on {H}amming {G}raphs.
\newblock arXiv:2103.15742, 2021.

\bibitem{bernard2021entanglement}
P.-A. Bernard, N.~Crampe, and L.~Vinet.
\newblock {E}ntanglement of {F}ree {F}ermions on {J}ohnson {G}raphs.
\newblock arXiv:2104.11581, 2021.

\bibitem{brouwer2012distance}
A.~E. Brouwer, A.~M. Cohen, and A.~Neumaier.
\newblock {\em Distance-{R}egular {G}raphs}.
\newblock Springer-Verlag, Berlin, 1989.

\bibitem{chan2019fractional}
A.~Chan, G.~Coutinho, C.~Tamon, L.~Vinet, and H.~Zhan.
\newblock {F}ractional {R}evival and {A}ssociation {S}chemes.
\newblock arXiv:1907.04729, 2019.

\bibitem{crampHad}
N.~Crampé, K.~Guo, and L.~Vinet.
\newblock {E}ntanglement of {F}ree {F}ermions on {H}adamard {G}raphs.
\newblock {\em Nuclear Physics B}, 960:115176, Nov 2020.

\bibitem{GAO2014164}
S.~Gao, L.~Zhang, and B.~Hou.
\newblock The {T}erwilliger algebras of {J}ohnson graphs.
\newblock {\em Linear Algebra and its Applications}, 443:164--183, 2014.

\bibitem{GAO2015427}
X.~Gao, S.~Gao, and B.~Hou.
\newblock {T}he {T}erwilliger algebras of {G}rassmann graphs.
\newblock {\em Linear Algebra and its Applications}, 471:427--448, 2015.

\bibitem{GO2002399}
J.~T. Go.
\newblock The {T}erwilliger {A}lgebra of the {H}ypercube.
\newblock {\em European Journal of Combinatorics}, 23(4):399 -- 429, 2002.

\bibitem{huang2021clebschgordan}
H.-W. Huang.
\newblock {T}he {C}lebsch-{G}ordan rule and the {H}amming graphs.
\newblock arXiv:2106.06857, 2021.

\bibitem{klimyk2012quantum}
A.~Klimyk and K.~Schm{\"u}dgen.
\newblock {\em {Q}uantum groups and their representations}.
\newblock Springer Science \& Business Media, 2012.

\bibitem{koekoek1996askey}
R.~Koekoek and R.~F. Swarttouw.
\newblock The {A}skey-scheme of hypergeometric orthogonal polynomials and its
  q-analogue.
\newblock Technical Report 98-17, Delft University of Technology, Faculty of
  Information Technology and Systems, Department of Technical Mathematics and
  Informatics, 1998.

\bibitem{LEEP1999157}
D.~B. Leep and L.~M. Schueller.
\newblock {Z}eros of a {P}air of {Q}uadratic {F}orms {D}efined over a {F}inite
  {F}ield.
\newblock {\em Finite Fields and Their Applications}, 5(2):157--176, 1999.

\bibitem{leonard1982orthogonal}
D.~A. Leonard.
\newblock {O}rthogonal polynomials, duality and association schemes.
\newblock {\em SIAM Journal on Mathematical Analysis}, 13(4):656--663, 1982.

\bibitem{LEVSTEIN20071621}
F.~Levstein and C.~Maldonado.
\newblock The {T}erwilliger algebra of the {J}ohnson schemes.
\newblock {\em Discrete Mathematics}, 307(13):1621--1635, 2007.

\bibitem{LEVSTEIN20061}
F.~Levstein, C.~Maldonado, and D.~Penazzi.
\newblock The {T}erwilliger algebra of a {H}amming scheme {H}(d,q).
\newblock {\em European Journal of Combinatorics}, 27(1):1 -- 10, 2006.

\bibitem{LIANG2020117}
X.~Liang, T.~Ito, and Y.~Watanabe.
\newblock {T}he {T}erwilliger algebra of the {G}rassmann scheme ${J}_q(n,d)$
  revisited from the viewpoint of the quantum affine algebra ${U}_q(sl_2)$.
\newblock {\em Linear Algebra and its Applications}, 596:117--144, 2020.

\bibitem{Obse}
X.~Liang, Y.-Y. Tan, and T.~Ito.
\newblock An {O}bservation on {L}eonard {S}ystem {P}arameters for the
  {T}erwilliger {A}lgebra of the {J}ohnson {S}cheme {J(N, D)}.
\newblock {\em Graph. Comb.}, 33(1):149–156, Jan. 2017.

\bibitem{macwilliams1969orthogonal}
J.~MacWilliams.
\newblock {O}rthogonal matrices over finite fields.
\newblock {\em The American Mathematical Monthly}, 76(2):152--164, 1969.

\bibitem{schmidt2006equations}
W.~M. Schmidt.
\newblock {\em {E}quations over finite fields: an elementary approach}, volume
  536.
\newblock Springer, 2006.

\bibitem{stanton1980some}
D.~Stanton.
\newblock {S}ome q-{K}rawtchouk polynomials on {C}hevalley groups.
\newblock {\em American Journal of Mathematics}, 102(4):625--662, 1980.

\bibitem{TAN2019157}
Y.-Y. Tan, Y.-Z. Fan, T.~Ito, and X.~Liang.
\newblock {T}he {T}erwilliger algebra of the {J}ohnson scheme {J(N,D)}
  revisited from the viewpoint of group representations.
\newblock {\em European Journal of Combinatorics}, 80:157--171, 2019.
\newblock Special Issue in Memory of Michel Marie Deza.

\bibitem{terwilliger1990incidence}
P.~Terwilliger.
\newblock The incidence algebra of a uniform poset.
\newblock In {\em Coding theory and design theory}, pages 193--212. Springer,
  1990.

\bibitem{thini}
P.~Terwilliger.
\newblock The {S}ubconstituent {A}lgebra of an {A}ssociation {S}cheme ({P}art
  {I}).
\newblock {\em J. Algebraic Comb.}, 1(4):363–388, Dec. 1992.

\bibitem{thinii}
P.~Terwilliger.
\newblock The {S}ubconstituent {A}lgebra of an {A}ssociation {S}cheme ({P}art
  {II}).
\newblock {\em J. Algebraic Comb.}, 2(1):73–103, Mar. 1993.

\bibitem{thiniii}
P.~Terwilliger.
\newblock The {S}ubconstituent {A}lgebra of an {A}ssociation {S}cheme ({P}art
  {III}).
\newblock {\em J. Algebraic Comb.}, 2(2):177–210, June 1993.

\bibitem{TERWILLIGER2003463}
P.~Terwilliger.
\newblock Introduction to {L}eonard pairs.
\newblock {\em Journal of Computational and Applied Mathematics},
  153(1):463--475, 2003.
\newblock Proceedings of the 6th International Symposium on Orthogonal Poly
  nomials, Special Functions and their Applications, Rome, Italy, 18-22 June
  2001.

\bibitem{watanabe2017algebra}
Y.~Watanabe.
\newblock An algebra associated with a subspace lattice over a finite field and
  its relation to the quantum affine algebra {U}${}_q(sl_2)$.
\newblock {\em Journal of Algebra}, 489:475--505, 2017.

\bibitem{WORAWANNOTAI2013443}
C.~Worawannotai.
\newblock {D}ual polar graphs, the quantum algebra ${U}_q(sl_2)$, and {L}eonard
  systems of dual q-{K}rawtchouk type.
\newblock {\em Linear Algebra and its Applications}, 438(1):443--497, 2013.

\end{thebibliography}
\end{document}